\documentclass [11pt,a4paper]{article}
\usepackage{amsfonts,amssymb,amsmath,amscd,latexsym,makeidx,theorem,color,hyperref,enumitem}
\usepackage[english]{babel}
\usepackage[utf8]{inputenc}

\author{Massimo Grossi \\ \small Sapienza Università di Roma  \\ \footnotesize \texttt{massimo.grossi@uniroma1.it} \and  \hspace{3cm} Gabriele Mancini \\ \hspace{3cm} \small  Sapienza Università di Roma \\ \hspace{3cm} \footnotesize \texttt{gabriele.mancini@uniroma1.it} \vspace{0.5cm} \and  Daisuke Naimen \\ \small Muroran Institute of Technology \\ \footnotesize \texttt{naimen@mmm.muroran-it.ac.jp} \and Angela Pistoia \\ \small Sapienza Università di Roma \\ \footnotesize \texttt{angela.pistoia@uniroma1.it}}

\title{Bubbling nodal solutions for a large perturbation of the Moser-Trudinger equation  on planar domains}
\newtheorem{trm}{Theorem}[section]
\newtheorem{prop}[trm]{Proposition}
\newtheorem{cor}[trm]{Corollary}
\newtheorem{lemma}[trm]{Lemma}

\newtheorem{rem}[trm]{Remark}
\newtheorem{Step}{Step}

\renewcommand{\(}{\left(}
\renewcommand{\)}{\right)}

\newcommand{\T}{\mathcal{T}}
\newcommand{\wt}{\widetilde}

\newcommand{\Om}{\Omega}
\newcommand{\de}{\delta}
\newcommand{\R}{\mathbb{R}}

\newcommand{\bra}[1]{\left({#1}\right)}

\newcommand{\eps}{\varepsilon}
\newcommand{\ph}{\varphi}
\newcommand{\ra}{\rightarrow}

\newcommand{\DD}[2]{\frac{\partial #1}{\partial #2}}

\newenvironment{proof}{\noindent\emph{Proof.}}{\phantom{ } \hfill$\square$\medskip}

\newcommand{\be}{\begin{equation*}}
\newcommand{\ee}{\end{equation*}}

\newcommand{\e }{\varepsilon }

\newcommand{\bdm}{\begin{displaymath}}
\newcommand{\edm}{\end{displaymath}}

\newcommand{\ov}{\overline}

\begin{document}
\maketitle

\begin{abstract}
In this work we study the existence of nodal solutions for the   problem 
$$
-\Delta u = \lambda u e^{u^2+|u|^p} \text{ in }\Omega, u = 0 \text{ on }\partial \Omega,
$$
where $\Omega\subseteq \R^2$ is a bounded smooth domain and  $p\to 1^+$. 


 If $\Omega$ is ball,  it is known that   the case $p=1$  defines a critical threshold between the existence and the non-existence of radially symmetric sign-changing solutions. In this work  we construct a blowing-up family of nodal solutions to such problem as $p\to 1^+$, when $\Omega$ is an arbitrary  domain and $\lambda$ is small enough.  As far as we know, this is the first construction of  sign-changing solutions for a Moser-Trudinger critical equation on a non-symmetric domain. 
\end{abstract}

\section{Introduction}
Let us consider the equation
 \begin{equation}\label{p}
\Delta u+\lambda ue^{u^2+a|u|^p}=0\ \hbox{in}\ \Omega,\ u=0\ \hbox{on}\ \partial \Omega,\end{equation}
where $\Omega$ is a bounded smooth domain in $\mathbb R^2$, $\lambda$ is a positive parameter and the nonlinear term $h(u):=ue^{ a|u|^p}$,  with $a\in\mathbb R$ and $p\in[0,2)$, is a lower-order perturbation of $e^{u^2}$  according to the definition given by Adimurthi in \cite{a}.
\medskip
 
The nonlinearity $f(u)=h(u)e^{u^2}$ is critical from the view point of the Trudinger imbedding. Indeed, in view of the Moser-Trudinger inequality (see \cite{Pohozaev, Trudinger,Moser}) 
\begin{equation}\label{mt_ine}
 \sup\left\{\int\limits_\Omega e^{ u^2} \,dx \ :\ u\in H^1_0(\Omega),\  \|u\|_{H^1_0(\Omega)}^2\le4\pi\right\}<+\infty,
\end{equation}
the functional 
\begin{equation}\label{j}
J_\lambda(u):=\frac12\int\limits_\Omega |\nabla u|^2dx-\lambda\int\limits_\Omega F(u)dx,\ u\in H_0^1(\Omega),
\end{equation}
where
$F(t)=\int\limits_0^t f(s)ds,$ is well defined and its critical points are solutions to problem \eqref{p}.
Adimurthi in \cite{a} proved that $J_\lambda$ satisfies the Palais-Smale condition   in the infinite energy range $(-\infty,2\pi)$ but, as observed by Adimurthi and Prashant in \cite{ap1}, the critical nature of $f(u)$ reflects in the failure of the Palais-Smale condition at the sequence of energy levels $2\pi k$ with $k\in\mathbb N$ (see also \cite{AS}). 

\medskip
In \cite{a} Adimurthi proved the existence of a critical point of $J_\lambda$
 if the perturbation $h$ is large, i.e. $a\ge0$, 
and if 
$0<\lambda  <\lambda_1(\Omega),$
 where $\lambda_1(\Omega)$ is the first eigenvalue of $-\Delta$ with Dirichlet boundary condition ((see also \cite{AdiPositive})). Such a critical point is a positive solution to problem \eqref{p}. Successively, Adimurthi and Prashant in \cite{ap2} showed that   the condition $a\ge0$ is   necessary   to get a positive solution to \eqref{p}. Indeed, they proved  that if the perturbation $h$ is small,  i.e. $a<0$,
  then there are no positive solutions to problem \eqref{p} when the domain $\Omega$ is a  ball provided $\lambda$ is small. {The  case $a=0$ in a general domain $\Omega$ has been studied by Del Pino, Musso and Ruf \cite{dmr} using a perturbative approach. Indeed they find  multiplicity of positive solutions which blow-up in one or more points of $\Omega$ (depending on the geometry) as $\lambda\to0.$  We point out that a general  qualitative analysis  of blowing-up families of positive solutions to problem \eqref{p} has been obtained by Druet in \cite{Druet} (see also \cite{AD,DT,dmmt}). }

\medskip
As far as it concerns the existence of sign-changing solutions, Adimurthi and Yadava in \cite{ay1} proved that problem \eqref{p} has a nodal solution when $\lambda$ is small if there is the further restriction $p>1$ on the growth of the large perturbation $h$ (i.e. $a>0$).
Actually, this condition turns out to be optimal for the existence of nodal radial solutions in a ball. Indeed Adimurthi and Yadava in \cite{ay2} proved that  if   $a>0 $ and $\Omega$ is a ball, problem \eqref{p} does not have any radial sign-changing solution when $\lambda$ is small and   $p\in[0,1]$. If one drops the radial requirement, Adimurthi and Yadava in \cite{ay1} proved the existence of infinitely many sign-changing solutions in a ball whatever $\lambda>0$ is.
We point out that, in the case $a=0$, the approach of Del Pino, Musso and Ruf \cite{dmr} allows  to find sign-changing solutions which blow-up positively and negatively at least at two different points in  any domain $\Omega$ as $\lambda\to0$ (even if this is not explicitly said in their work). 

\medskip
According to the previous discussion, it turns out that when $a>0$  the case $p=1$
defines  a critical threshold  for the existence of radial   sign-changing solutions in the ball. Indeed, 
when $\Omega= B(0,1)$,  \eqref{p} has radially symmetric sign-changing solutions which blow-up as $p\to{1^+} $. The precise behavior of such solutions was studied by Grossi and Naimen in \cite{GN}. 
Therefore, { when $a>0$, it is natural to ask whether  it is possible to find sign-changing solutions to problem \eqref{p} on an arbitrary planar domain $\Omega$ which blow-up at one point in $\Omega$  as $p\to{1^+}$}.


\medskip
In this paper we give a positive answer. More precisely, let us consider the problem
\begin{equation}\label{problem feps}
\begin{cases}
-\Delta u = \lambda u e^{u^2+|u|^{1+\eps}}& \text{ in }\Omega,\\
 u  = 0 & \text{ on }\partial \Omega,
\end{cases} 
\end{equation}
where $\eps$ is a positive small parameter.
Set
\begin{equation}\label{feps}
f_{\eps}(t) = t e^{t^2+|t|^{1+\eps}}.
\end{equation}

For a given $0<\lambda<\lambda_1(\Omega)$,  let $u_0$ be a positive solution of the problem  
\begin{equation}\label{equ0}
\begin{cases}
-\Delta u_0 = \lambda f_0(u_0) & \text{ in }\Omega,\\
u_0>0 & \text{ in }\Omega,\\
u_0 = 0 & \text{ on } \partial \Omega,
\end{cases}
\end{equation}
whose existence has been established by Adimurthi in \cite{a}.
We make the following assumptions:
\begin{enumerate}[label=(A\arabic*)]
\item \label{A1} $u_0$ is non-degenerate, i.e. there is no non-trivial solution $\ph \in H^1_0(\Omega)$ of the equation 
\begin{equation}
-\Delta \ph = \lambda f'_0(u_0) \ph\ \hbox{in}\ \Omega,\ \ph=0\ \hbox{on}\ \partial\Omega.
\end{equation}
\item \label{A2} $ u_0$ has a $C^1-$stable critical point $\xi_0\in \Omega$ such that $u_0(\xi_0)>\frac{1}{2}$.   
\end{enumerate}
Then, we will show that \eqref{problem feps} admits a family of  sign-changing solutions which blow-up at $\xi_0$ with residual mass $-u_0$  as $\eps \to 0$, namely:
\begin{trm}\label{main result}
For $0<\lambda< \lambda_1(\Omega)$, let $u_0$ be a solution of \eqref{equ0} such that \ref{A1} and \ref{A2} are satisfied. Let also $\xi_0$ be as in \ref{A2}. Then there exist $\eps_0>0$  and a family $(u_\eps)_{0<\eps<\eps_0}$ of sign-changing solutions to \eqref{problem feps} such that:
\begin{itemize}
\item $\displaystyle{\max_{\ov{B(\xi_0,r)}} u_\eps \to +\infty}$ as $\eps \to 0$, for any $0<r<d(\xi_0,\partial \Omega)$. 
\item $u_\eps \to - u_0$ weakly in $H^1_0(\Omega)$ and in $C^1(\ov \Omega \setminus \{\xi_0\})$.
\end{itemize}
\end{trm}

Let us make some comments about assumtpions \ref{A1} and \ref{A2}.
\begin{rem}
\begin{itemize}
\item The solution $u_0$ to problem \eqref{equ0} turns out to be non-degenerate when $\Omega$ is the ball as proved by Adimurthi, Karthik and Giacomoni in \cite{akg}.
In a work in progress, Grossi and Naimen are going to prove that the solution is also non-degenerate when $\Omega$ is convex and symmetric (see \cite{gn2}). Actually, we believe that the non-degeneracy condition holds true for most domains $\Omega$ and positive parameters $\lambda.$ Indeed, one could use similar arguments to those used   by Micheletti and Pistoia in \cite{mipi}   for a class of singularly perturbed equations.
\item We remind that $\xi_0$ is a $C^1-$stable critical point  of $u_0$ if the Brouwer degree ${\rm deg}\(\nabla u_0, B(\xi_0,r),0\)\not=0.$ In particular, any strict local maximum point of $u_0$ is $C^1-$stable. We point out that by Adimurthi and Druet \cite{AD} we can deduce that assumption \ref{A2} holds true when the parameter $\lambda$ is small enough.
\item We strongly believe that the condition $u_0(\xi_0)>\frac12$ is not purely technical, but it is necessary to build a solution which blows-up at $\xi_0.$ Indeed, we conjecture that, if $u_0(\xi_0)\le\frac12$, there does not exist any sign-changing solution which blows-up at $\xi_0$ with non-trivial residual mass $u_0$ as $\eps\to0.$ 
We point out that, in a different setting,  a similar condition was proved by Mancini and Thizy \cite{ManciniThizy} for  problem \eqref{p} on a ball with $p=1$ and $a<0$: in fact, they show  that the value at the origin of the residual mass of any non-compact sequence of radially symmetric positive solutions must be equal to  $-\frac{a}{2}$ (and we get $\frac{1}{2}$, when $a=-1$).
\end{itemize}
\end{rem}

\bigskip

Actually, we can give a more precise description of the asymptotic behavior of the solution $u_\eps$ as $\eps\to0,$ since it is build via  a Lyapunov-Schmidt procedure.  For $\delta,\mu>0$, and $\xi \in \R^n$, let us consider the functions
\begin{equation}\label{bubbles}
U_{\delta,\mu,\xi}(x)= \log\left( \frac{8\mu^2 \delta^2}{(\mu^2 \delta^2 +|x-\xi|^2)^2}\right),
\end{equation}
which describe the set of all the solutions to the Liouville equation
\begin{equation}\label{Liouville}
-\Delta U = e^{U} \quad \text{in } \R^2,
\end{equation}
under the condition $e^{U}\in L^1(\R^2)$ (see \cite{Lio,CL}). We further consider the projection $P U_{\delta,\mu,\xi}:= (-\Delta)^{-1} e^{U_{\de,\mu,\xi}}$, where $(-\Delta)^{-1}:L^{2}(\Omega) \ra H^1_0(\Omega)$ is the inverse of $-\Delta$. Namely, $P U_{\delta,\mu,\xi}$ is defined as the unique solution to 
\begin{equation}\label{0}
\begin{cases}
-\Delta PU_{\de,\mu,\xi}=-\Delta U_{\de,\mu,\xi} = e^{U_{\delta,\mu,\xi}}
&\text{ in }\Om,\\
PU_{\de,\mu,\xi}=0&\text{ on }\partial\Om.
\end{cases}
\end{equation}
Intuitively, we want to look for solutions of \eqref{problem feps} that look like 
$\alpha P U_{\delta,\mu,\xi}- u_0$ for suitable choices of the parameters $\alpha,\delta,\mu,\xi$. Unfortunately, in order to succesfully perform Lyapunov-Schmidt reduction, a more precise ansatz is necessary and we are forced to replace $u_0$ with a better approximation of the solutions. First, the non-degeneracy assumption \ref{A1} allows to find a positive solution $v_\eps\in C^1(\ov \Omega)$ of \eqref{problem feps} 
such that
$$
v_\eps \to u_0 \qquad \text{ in } C^1(\ov \Omega),
$$
as $\eps\to 0$. Then, we consider the function 
\begin{equation}\label{Veps}
V_{\eps,\alpha,\xi}: = v_\eps +\alpha w_{\eps,\xi} + \alpha^2 z_{\eps,\xi},
\end{equation}
where $\alpha\in (0,1)$ is a small positive parameter depending on $\eps,\mu,\xi$ such that $\alpha \to 0$ as $\eps\to 0$, and $w_{\eps,\xi}$ and $z_{\eps,\xi}$ are defined as the unique solutions to the  couple of linear problems
\begin{equation}\label{eqweps}
\begin{cases}
\Delta w_{\eps,\xi} + \lambda f_\eps'(v_\eps) w_{\eps,\xi} = 8\pi \lambda G_\xi f'_\eps(v_\eps) & \text{ in }\Omega,\\
w_{\eps,\xi} = 0  & \text{ on }\partial \Omega,
\end{cases}
\end{equation}
and 
\begin{equation}\label{eqzeps}
\begin{cases}
\Delta z_{\eps,\xi} +\lambda f_\eps'(v_{\eps}) z_{\eps,\xi} =  \frac{\lambda}{2} f''_\eps(-v_\eps)(8\pi G_\xi - w_\eps)^2 & \text{ in }\Omega, \\
z_\eps = 0 & \text{ on }\partial \Omega,
\end{cases}
\end{equation} 
with $G_\xi$ denoting the Green function of $\Omega$ with singularity at $\xi$, namely the distributional solution to
\begin{equation}\label{Green}
\begin{cases}
-\Delta G_\xi = \delta_\xi  & \text{ in }\Omega,  \\
G_\xi = 0  & \text{ on } \partial \Omega.
\end{cases}
\end{equation}
Problems \eqref{eqweps} and \eqref{eqzeps} are nothing but  the   linearization of problem \eqref{problem feps} around the solution $v_\eps$ and  the R.H.S.'s are the terms of the second order Taylor's expansion with respect to $\alpha$ of $f_\eps(\alpha PU_{\delta,\mu,\xi}-{V_{\eps,\alpha,\xi}})$ far away from the concentration point $\xi$
(indeed $PU_{\delta,\mu,\xi}\sim  8\pi G_\xi$ because of \eqref{green}).

\medskip
Theorem \ref{main result} follows at once by the following result:

\begin{trm}\label{Trm precise}
Let $\lambda$, $u_0$, $\xi_0$ be as in Theorem \ref{main result}. There exists $\epsilon_0>0$ and   functions $\alpha,\delta, \mu:(0,\eps_0) \to (0,+\infty)$, $\xi:(0,\eps_0)\ra \Omega$ and $\ph:(0,\eps_0)\ra H^1_0(\Omega)$   such that: 
\begin{itemize}
\item $u_\eps:= \alpha(\eps) P U_{\delta(\eps), \mu(\eps),\xi(\eps)}  - V_{\eps,\alpha(\eps),\xi(\eps)} + \ph(\eps) $  is a solution \eqref{problem feps}. 
\item $\alpha(\eps) \to 0$, $\delta(\eps) \to 0$, $\mu(\eps)\to \sqrt{8} e^{-1}$, $\xi(\eps)\to \xi_0$, and $u_\eps(\xi(\eps))\to +\infty$ as $\eps \to 0$.  
\item $\|\ph(\eps)\|_{H^1_0(\Omega)} + \|\ph(\eps)\|_{L^\infty(\Omega)} = O(e^{-\frac{\log (2 u_0(  {\xi_0} ))}{\eps}})$. 
\end{itemize} 
\end{trm}
\bigskip

Let us briefly sketch the main steps of 
 the proof of Theorem \ref{Trm precise}. First, in Section \ref{sec1}, we choose $\alpha = \alpha(\eps,\mu,\xi)$ and $\delta = \delta(\eps,\mu,\xi)$ such that the function
\begin{equation}\label{approximatesol}
\omega_{\eps,\mu,\xi} := \alpha PU_{\delta,\mu,\xi} -V_{\eps,\alpha,\xi}
\end{equation}
is an approximate solution of \eqref{problem feps}. Then, we look for solutions of \eqref{problem feps}  of the form $\omega_{\eps,\mu,\xi} + \ph$ with $\ph \in H^1_0(\Omega)$. Clearly,  \eqref{problem feps} can be written in terms of $\ph$ as 
\begin{equation}\label{eqphi}
-\Delta \ph - \lambda f_\eps '(\omega_{\eps,\mu,\xi}) \ph = R + N(\ph), 
\end{equation} 
where the error  term $R$ is defined by
\begin{equation}\label{DefR}
R= R_{\eps,\mu,\xi} := \Delta \omega_{\eps,\mu,\xi} + \lambda f_\eps(\omega_{\eps,\mu,\xi }),
\end{equation}
and the higher order term $N$  by
\begin{equation}\label{DefN}
N(\ph)=N_{\eps,\mu,\xi }( \ph):=  \lambda \left( f_\eps(\omega_{\eps,\mu,\xi} + \ph) - f_\eps(\omega_{\eps,\mu,\xi}) - f_\eps'(\omega_{\eps,\mu,\xi}) \ph \right) .
\end{equation}
Equivalently, introducing the linear operator  
\begin{equation}\label{DefL}
L\ph = L_{\eps,\mu,\xi }\ph:= \ph - (-\Delta)^{-1} (\lambda f'(\omega_{\eps,\mu,\xi})   \ph),
\end{equation}
we need to solve
\begin{equation}\label{eqph2}
 L \ph= (-\Delta)^{-1} \left( R + N(\ph)\right).
\end{equation}
A careful and delicate estimate of the error $R$ will be given   in Section  \ref{Sec R}. The behaviour of the operator $L$ will be studied in Section \ref{Sec L}. On the one hand, for functions supported away from a suitable schrinking neighborhood of $\xi$, we will show that $L$ is close to the operator $L_1 \ph :=  \ph - (-\Delta)^{-1} (\lambda f'_0(u_0)  \ph)$, which is  invertible on $H^1_0(\Omega)$  because  of the non-degeneracy assumption \ref{A1}.   On the other hand, near the point $\xi$, $L$ is close to the operator $L_{0}\ph:= \ph - (-\Delta)^{-1} (e^{U_{\delta,\mu,\xi}}\ph )$. This operator  appears in the analysis of several critical problems in dimension $2$ (see for example \cite{BP,dkm,egp}) and its behavior is well known:  although $L_0$ is not invertible, it is possible to find an approximate three-dimensional kernel $K_{\de,\mu,\xi}$ for $L_0$ by projecting on $H^1_0(\Omega)$ the three functions
$$
Z_{0,\de,\mu,\xi} (x) = \frac{\delta^2\mu^2 -|x-\xi|^2}{|x-\xi|^2+ \delta^2 \mu^2},  \qquad Z_{i,\de,\mu,\xi} (x) = \frac{2 \de \mu (x_i-\xi_i)}{|x-\xi|^2+ \delta^2 \mu^2}, 
\qquad i=1,2. 
$$ 
Such properties transfer to the operator $L$, which turns out to be invertible on the subspace $K_{\de,\mu,\xi}^\perp$   orthogonal to $K_{\de,\mu,\xi}$ in $H^1_0(\Omega)$. More precisely, denoting  by  $\pi$ and $\pi^{\perp}$ the projections of $H^1_0(\Omega)$ respectively on  $K_{\de,\mu,\xi}$ and $K_{\delta,\mu,\xi}^\perp$, we will show that $ \pi^\perp L$ is invertible on $K_{\delta,\mu,\xi}^{\perp}$. Then, it is  natural to split equation \eqref{eqph2} as 
\begin{equation}\label{system}
\begin{cases}
\ph = (\pi^\perp  L)^{-1}\pi^\perp  \left(- \Delta \right)^{-1} \left( R + N(\ph) \right), \\
\pi  L\ph = \pi    \left(- \Delta \right)^{-1} \left( R + N(\ph) \right) .
\end{cases}
\end{equation} 
The first equation of \eqref{system} will be solved in Section \ref{Sec T}, where for any $\mu>0$, $\xi$ close to $\xi_0$ and any small $\eps>0$, we will   find a solution $\ph_{\eps,\mu,\xi}$   via a contraction mapping argument on a sufficiently small ball in   $K_{\de,\mu,\xi}^{\perp}\cap L^\infty(\Omega)$.  Then, recalling that $\dim K_{\de,\mu,\xi}=3$ and using assumption \ref{A2}, we will show in Section \ref{Sec Par} that it is possible to choose the three parameters $\mu = \mu(\eps)$ and $\xi = \xi (\eps) = (\xi_1(\eps),\xi_2(\eps))$ so that  the second equation in \eqref{system} is also fullfilled.   Clearly, for such choice of $\mu$ and $\xi$, the function $\ph_{\eps,\mu(\eps),\xi(\eps)}$ solves both the equations in \eqref{system} (or, equivalently \eqref{eqphi} and \eqref{eqph2}), and 
$u_\eps: ={\omega_{\eps,\mu(\eps),\xi(\eps)}+\ph_{\eps,\mu(\eps),\xi(\eps)}}
$ is a solution of \eqref{problem feps}. 

\medskip
{It is important to point out that choice of the concentration point $\xi(\eps)$ is extremely delicate since the scaling parameter $\delta$ turns out to be much smaller than the parameter $\alpha$, whose powers control all the error terms. To overcome this difficulty, we introduce a new argument based on a precise Pohozaev-type identity. This allows us to bypass global a priori gradient estimates on the solution $\ph_{\eps,\mu,\xi}$, which are hard to obtain for Moser-Trudinger critical problems.  Our argument requires a very precise ansatz of the approximate solution $\omega_{\eps,\mu,\xi}$. In particular, the presence of the correction terms $w_{\eps,\xi}$ and $z_{\eps,\xi}$ in the expression of $V_{\eps,\alpha,\xi}$ is not merely technical, but plays a crucial role both in the estimates of the error term $R$ and in the choice of $\xi(\eps)$. 
}

\section{Construction of the approximate solution}\label{sec1}
In this section we give the detailed construction of the approximate solution $\omega_{\eps,\mu,\xi}$. Here and in the rest of the paper, we will assume that $(\mu,\xi)\in \mathcal{U}\times B(\xi_0,\sigma)$, where $ \mathcal U \Subset \R^{+}$ is an open interval containing ${ \mu_0:=\sqrt{8}e^{-1}}$, $\xi_0$ is as in the assumption \ref{A2},  and $0<\sigma<\frac{1}{2} d(\xi_0,\partial \Omega)$. By \ref{A2}, we can also assume
\begin{equation}\label{sigma small}
\inf_{B(\xi_0,\sigma)} u_0(\xi) >\frac{1}{2}. 
\end{equation}

\subsection{The main terms of the ansatz}
Let us introduce the main property of the projection of the bubble $PU_{\de,\mu,\xi }$ defined in \eqref{0}, which gives the main term of the approximate solution close to the blow-up point $\xi.$ Let  $G_\xi(\cdot)=G(\cdot,\xi)$  be the Green's function  of $-\Delta$ with Dirichlet boundary conditions introduced in \eqref{Green}
 and let  $H(\cdot,\xi)$   be its regular part, i.e.
$$
H(x,\xi) := G_\xi(x )  - \frac{1}{2 \pi} \log \frac{1}{|x-\xi|}.
$$

\begin{lemma}\label{3}
We have
\[
\begin{split}
PU_{\de,\mu,\xi}(x)=&U_{\de,\mu,\xi }(x)-\log(8\mu^2\de^2)+8\pi H(x,\xi)+\psi_{\de,\mu,\xi}(x),
\end{split}
\]
where $$\|\psi_{\de,\mu,\xi }\|_{C^{1}(\ov\Omega)} = O(\de^2),$$
uniformly with respect to $\mu \in \mathcal U $, $\xi \in B(\xi_0,\sigma)$.   
\\
In particular,
\begin{equation}\label{green}
PU_{\de,\mu,\xi}\to  8\pi G_\xi\ \hbox{in}\ C^{1}_{loc}(\ov\Omega\setminus\{\xi\}).
\end{equation}
\end{lemma}
\begin{proof} See for example {\cite[Proosition 5.1]{BartP}}. 
\end{proof}

Next, let us define the main term of the approximate solution in the whole domain as
$\alpha PU_{\de,\mu,\xi}-v_\eps$ where $\alpha$ is a positive parameter approaching zero as $\eps\to0$ and $v_\eps$ is a non-degenerate solution to \eqref{problem feps}, whose existence is proved in the following lemma.
 
\begin{lemma}\label{Lemma veps} Let $\lambda$ and $u_0$ be as in Theorems \ref{main result} and \ref{Trm precise}.  There exists $\eps_0>0$, and a family of functions $(v_\eps)_{0<\eps<\eps_0} \subseteq {C^{1}(\ov \Omega) }$ such that: 
\begin{enumerate}[label=\roman*.]
\item\label{veps i} $v_\eps $ is a non-degenerate weak solution of \eqref{problem feps} for any $\eps \in (0,\eps_0)$.
\item\label{veps ii} $v_\eps \to u_0 \; \text{ in } {C^{1}(\ov \Omega)}$ as $\eps \to 0$. 
\item\label{veps iii} There exists $c>0$ such that $v_\eps(x) \ge c d(x,\partial \Omega)$ for any $x\in \Omega$, $\eps \in (0,\eps_0)$.  
\end{enumerate}
\end{lemma}
\begin{proof}
Let $F :(-1,1)\times H^1_0(\Omega)\to H^1_0(\Omega)$ be defined by
\begin{equation}\label{Fe}
F(\eps,u)=F_\eps(u):= u -(-\Delta)^{-1} (\lambda f_\eps(u)),
\end{equation}
where  $f_\eps$ is defined as in \eqref{feps}. $F$ is well defined because
  the Moser-Trudinger inequality  \eqref{mt_ine} implies that $f_\eps(u)\in L^p(\Omega)$ for any $1\le p<+\infty$ and $u\in H^1_0(\Omega)$. Moreover, it is a $C^1$-map and its   partial derivative    $DF_\eps (u):H^1_0(\Omega) \to H^1_0(\Omega)$  defined by
$$
DF_\eps(u)[\ph] = \ph - (-\Delta)^{-1} (\lambda f_\eps'(u)\ph)$$
 is a Fredholm operator of   index $0$ (since the embedding {$H^1_0(\Omega) \hookrightarrow L^p(\Omega)$} is compact).

Now, let $u_0$ be  a non-degenerate weak solution  of \eqref{equ0} such that \ref{A1} holds true.   In particular,   $F_0(u_0)=0$ and $DF_0(u_0)$ is invertible. Therefore, by the implicit function theorem, we can construct a $C^1$ curve $\eps \mapsto v_\eps\in H^1_0(\Omega)$, defined for $|\eps|<\eps_0$ such that $v_0 = u_0$, $F_\eps(v_\eps) =0 $, and $DF_\eps(v_\eps)$ is invertible for  $|\eps|<\eps_0$. Then \emph{\ref{veps i}} holds. 

Applying the Moser-Trudinger inequality \eqref{mt_ine} and standard elliptic estimates, we obtain \emph{\ref{veps ii}}. 

Hopf's lemma and the compactness of $\partial \Omega$ give $\DD{u_0}{\nu}\le -2c$ on $\partial \Omega$, for some $c>0$. Then, for $\eps$ sufficiently small, we have $\DD{v_\eps}{\nu}\le -c$, which in turn gives $v_\eps (x) \ge c d(x,\partial \Omega)$ for $x$ in a neighborhood of $\partial \Omega$. Finally, since $v_\eps \to u_0$ uniformly in $\ov \Omega$, and $u_0 >0$ in $\Omega$, we get \emph{\ref{veps iii}}.  
\end{proof}

\subsection{The correction of the ansatz}

We need to correct the ansatz in the whole domain by solving the following two  linear problems \eqref{eqweps} and \eqref{eqzeps}:
$$\begin{cases}
\Delta w_{\eps,\xi} + \lambda f_\eps'(v_\eps) w_{\eps,\xi} = 8\pi \lambda G_\xi f'_\eps(v_\eps) & \text{ in }\Omega,\\
w_{\eps,\xi} = 0  & \text{ on }\partial \Omega,
\end{cases}
$$
and
$$
\begin{cases}
\Delta z_{\eps,\xi} +\lambda f_\eps'(v_{\eps}) z_{\eps,\xi} =  \frac{\lambda}{2} f''_\eps(-v_\eps)(8\pi G_\xi - w_\eps)^2 & \text{ in }\Omega, \\
z_\eps = 0 & \text{ on }\partial \Omega.
\end{cases}
$$

\begin{lemma}\label{Lemma weps & zeps}
For any $0<\eps < \eps_0$  and any $\xi \in  \Omega$, there exist $w_{\eps,\xi}$, $z_{\eps,\xi}$ such that \eqref{eqweps} and \eqref{eqzeps} hold. {Moreover, there exists $C>0$ such that 
\begin{equation}\label{stima}
\|w_{\eps,\xi}\|_{C^{1}(\ov\Omega)}+ \|z_{\eps,\xi}\|_{C^{1}(\ov \Omega)} \le C 
\end{equation}
for $\eps \in (0,\eps_0)$, $\xi \in \Omega$. }
\end{lemma}
\begin{proof} The existence of the solutions immediately follows from the non-degeneracy of the function $v_\eps$ proved in Lemma \ref{Lemma veps}. Moreover, 
since for any $p\in [1,+\infty)$ one has 
$$
\sup_{\xi \in \Omega} \|G_{\xi}\|_{L^p(\Omega)} < +\infty \quad \text{ and } \quad  \sup_{0<\eps <\eps_0}\|v_\eps\|_{C^1(\ov \Omega)} <+\infty,
$$ 
 \eqref{stima} follows by standard elliptic estimates. 
\end{proof}

Finally, we introduce the corrected  ansatz as 
\begin{equation}
\omega_{\eps,\mu,\xi} := \alpha PU_{\delta,\mu,\xi} -  V_{\eps,\alpha,\xi} 
\end{equation}
with 
\begin{equation}
V_{\eps,\alpha,\xi}: = v_\eps +\alpha w_{\eps,\xi} + \alpha^2 z_{\eps,\xi},
\end{equation}
where $v_\eps$ is defined in Lemma \ref{Lemma veps} and $w_{\eps,\xi}$ and $z_{\eps,\xi}$ as in Lemma \ref{Lemma weps & zeps}.

\subsection{The choice of parameters}

It will be necessary to choose the parameters $\alpha=\alpha(\eps,\mu,\xi)$ and $\delta=\delta(\eps,\mu,\xi)$ such that 
$\lambda f_\eps(\omega_{\eps,\mu,\xi})\sim \alpha e^{U_{\delta,\mu,\xi}}$ when ${|x-\xi|}\sim \delta.$ We point out that one of the main difficulties in this problem is that this estimates holds true only at a very small scale.

\medskip

Let us fix the values of $\alpha$ and $\delta$ according to the next lemma. The proof is based on the {contraction mapping theorem} and is postponed to the appendix.

\begin{lemma}\label{Lemma parameters}
There exist  $\eps_0>0$ and  functions $\alpha =  \alpha(\eps,\mu,\xi)$, $\beta = \beta(\eps,\mu,\xi)$ and $\delta = \delta(\eps,\mu,\xi)$, defined in $(0,\eps_0)\times\mathcal U \times B(\xi_0,\sigma)$  and continuous with respect to $\mu$ and $\xi$,    such that 
\begin{equation}\label{parameters}
\begin{cases}
\lambda\beta e^{\beta^2+\beta^{1+\eps}} = \frac{\alpha}{\delta^2},\\
2\alpha \beta + \alpha \beta^{\eps} +\eps \alpha \beta^\eps = 1,\\
\beta= 4\alpha\log \frac{1}{\delta} - V_{\eps,\alpha,\xi}(\xi)+ \alpha c_{\mu,\xi},
\end{cases}
\end{equation}
where $c_{\mu,\xi}:= -\log (8\mu^2) + 8\pi H(\xi,\xi)$ and $V_{\eps,\alpha,\xi}$ is defined in \eqref{Veps}.

Moreover, as $\eps\to 0$,  we have that 
\begin{equation}\label{alpha}
\alpha(\eps,\mu,\xi) = \frac{1}{2} e^{-\frac{\log(2u_0(\xi))+o(1)}{\eps}},
\end{equation}
\begin{equation}\label{beta}
\beta(\eps,\mu,\xi) = \frac{1}{2\alpha} - u_0(\xi) + o(1),
\end{equation}
\begin{equation}\label{delta}
\log \frac{1}{\delta(\eps,\mu,\xi)} = \frac{1+o(1)}{8\alpha^2},
\end{equation}
where $o(1)\to 0$ as $\eps \to 0$, uniformly for $\mu \in \mathcal U$ and $\xi \in B(\xi_0,\sigma)$. 
\end{lemma}

\begin{rem}
Note that \eqref{alpha}-\eqref{delta} and \eqref{sigma small} give  $\alpha(\eps,\mu,\xi),\delta(\eps,\mu,\xi)\to 0$ and $\beta(\eps,\mu,\xi)\to +\infty$ as $\eps \to 0$, uniformly for $\mu\in \mathcal U$ and $\xi \in B(\xi_0,\sigma)$. 
\end{rem}

From now on we let $\alpha = \alpha (\eps,\mu,\xi)$, $\beta=\beta(\eps,\mu,\xi)$ and $\delta = \delta(\eps,\mu,\xi)$ be as in Lemma \ref{Lemma parameters}.  

It will be convenient to work on the scaled domain 
 $
\frac{\Omega -\xi}{\delta} :=\left\{ \frac{x-\xi}{\delta}, \quad x\in\Omega \right\}.  
$
Note that we have the scaling relation 
\begin{equation}\label{scaling}
U_{\de,\mu,\xi}(x) = \bar U_\mu \bra{ \frac{x-\xi}{\de}} - 2\log \delta,
\end{equation} 
where
\begin{equation}\label{Ubar}
\bar U_\mu (y)= U_{1,\mu,0}(y) = \log \left( \frac{8\mu^2}{(\mu^2+|y|^2)^2}\right).
\end{equation}

\begin{lemma} \label{expapproxsol} As $\eps\to 0$, we have
\begin{equation}\label{cru1}
\omega_{\eps,\mu,\xi}(\xi+ \delta y ) = \beta +\alpha \bar U_{\mu}(y) + O(\de|y|) + O(\delta^2),
\end{equation}
uniformly for $y\in B(0,\frac{\sigma}{\de})$, $\mu \in \mathcal U$ and $\xi \in B(\xi_0,\sigma)$. \\
Moreover, for any $R>0$  it holds also true that
\begin{equation}\label{motivation3}
\begin{split}
\lambda f_\eps(\omega_{\eps,\mu,\xi})(\xi + \delta y)  = \alpha e^{U_{\delta,\mu,\xi}(\xi+ \delta y)} (1+O(\alpha^2)),
\end{split}
\end{equation}
 as $\eps\to 0$ uniformly for  $y\in B(0,R)$, $\mu\in \mathcal U$ and $\xi \in B(\xi_0,\sigma)$. 
\end{lemma}

\begin{proof}  Lemma \ref{3} and the scaling relation \eqref{scaling} show that,
as $\de\to 0$, we have the following expansion uniformly for $\eps \in (0,\eps_0)$,  $\mu\in \mathcal U$,  $\xi \in B(\xi_0,\sigma)$ and $y\in B(0,\frac{\sigma}{\delta})$: 
\[
\begin{split}
\omega_{\eps,\mu,\xi}  (\xi + \delta y)  &  =  \alpha \bar U_{\mu} +\underbrace{ 4\alpha \log \frac{1}{\delta}  + \alpha c_{\mu,\xi} - V_{\eps,\alpha,\mu}(\xi)}_{=\beta} +V_{\eps,\alpha,\mu}(\xi) -  V_{\eps,\alpha,\xi}(\xi+\de y) \\
&\quad + 8\pi \alpha (  H(\xi+\de y,\xi) -H(\xi,\xi))  + O(\delta^2).
\end{split}
\]
By Lemmas \ref{Lemma veps} and \ref{Lemma weps & zeps}, we know that $V_{\eps,\alpha,\mu}$ is uniformly bounded in $C^{1}(\ov \Omega)$.   Thus
\[
V_{\eps,\alpha,\mu}(\xi +\delta y )= V_{\eps,\alpha,\mu}(\xi) + O(\delta |y|).
\]
Similarly, since $H\in C^{1}(\ov \Omega \times B(\xi_0,\sigma))$, we have 
$$
H(\xi+\delta y,\xi )=H(\xi,\xi)+ O(\delta |y|).
$$
Then  estimate \eqref{cru1} is proved. 

\medskip
Now, let us prove \eqref{motivation3}.
Note that \eqref{alpha}-\eqref{delta} yield $\beta = O(\frac{1}{\alpha})$, $\delta = O(e^{-\frac{1+o(1)}{8\alpha^2}})$, and $\beta^\eps = 2u_0(\xi)+o(1) = O(1)$.  For $|y|\le R$, \eqref{cru1} implies 
\[
\omega_{\eps,\mu,\xi}(\xi + \delta y) = \beta+\alpha \bar U_{\mu}(y) + O(\delta).
\]
In particular
\begin{equation}\label{motivation1}
\omega_{\eps,\mu,\xi}(\xi + \delta y)^2  = \beta^2 +2\alpha\beta \bar U_{\mu}(y) + O(\beta \delta),
\end{equation}
and
\begin{equation}\label{motivation2}
\begin{split}
\omega_{\eps,\mu,\xi}(\xi+ \delta y )^{1+\eps} & = (\beta + \alpha \bar U_{\mu}(y) + O(\delta)) (\beta + \alpha \bar U_{\mu}(y) + O(\delta))^{\eps} \\
&= (\beta + \alpha \bar U_{\mu}(y) + O(\delta)) \beta^\eps \left(1 + \frac{\alpha}{\beta} \bar U_{\mu}(y) + O(\alpha \delta)\right)^{\eps} \\
& =  \left( \beta^{1+\eps} + \alpha \beta^{\eps} \bar U_{\mu}(y) + O(\delta) \right)\left(1+\frac{\eps \alpha}{\beta} \bar U_{\mu}(y) + O(\eps \alpha^4)\right)\\
& = \beta^{1+\eps} + \alpha \beta^{\eps} \bar U_{\mu} (y) + \eps \alpha \beta^{\eps} \bar U_{\mu}(y) + O(\eps \alpha^3).
\end{split}
\end{equation}
Then, using \eqref{parameters} we get 
$$\begin{aligned}
\lambda f_\eps(\omega_{\eps,\mu,\xi})(\xi + \delta y) & = \lambda \omega_{\eps,\mu,\xi}(\xi+ \delta y) e^{\omega_{\eps,\mu,\xi }(\xi+ \delta y )^2 + \omega_{\eps,\mu,\xi }^{1+\eps}(\xi+ \delta y) } \\
&= \lambda \beta (1+O(\alpha^2)) e^{\beta^2 +\beta^{1+\eps}+ (2\alpha \beta+ \alpha\beta^{\eps} + \alpha\eps \beta^{\eps}) \bar U_{\mu}(y) + O(\alpha^2)}\\ 
&= \underbrace{\lambda \beta e^{\beta^2 +\beta^{1+\eps}}}_{=\frac{\alpha}{\delta^2}}  e^{ \underbrace{(2\alpha \beta+ \alpha\beta^{\eps} + \alpha\eps \beta^{\eps})}_{=1} \bar U_{\mu}(y)}(1+O(\alpha^2)) e^{ O(\alpha^2)}\\ &  = \frac{\alpha}{\delta^2} e^{\bar U_{\mu}(y)}(1+ O(\alpha^2)) \\
&= \alpha e^{U_{\delta,\mu,\xi}(\xi+ \delta y)} (1+O(\alpha^2)),
\end{aligned}
$$
which proves \eqref{motivation3}.
\end{proof}

It is also useful to point out the following result which will be used in the next sections.
\begin{rem}\label{Rem beta}
Lemma \ref{3} and Lemma \ref{Lemma parameters} give $$0\le \alpha P U_{\delta,\mu,\xi}\le \beta + u_{0}(\xi) + o(1),$$ and
$$-V_{\alpha,\eps,\xi}\le \omega_{\eps,\mu,\xi}\le \beta + o(1),$$ uniformly for $x\in\Omega$, $\eps \in (0,\eps_0)$, $\mu \in \mathcal{U}$, $\xi\in B(\xi_0,\sigma)$. 
\end{rem}

\bigskip
\emph{Notation:} In order to simplify the notation, we will write $U_{\eps}$,  $\bar  U $, $V_\eps$, $\omega_{\eps}$, $w_\eps$ and $z_\eps$ instead of $U_{\delta,\mu,\xi}$, $\bar U_{\mu}$, $V_{\eps,\alpha,\xi}$, $\omega_{\eps,\mu,\xi}$, $w_{\eps,\xi}$ and $z_{\eps,\xi}$, without specifying explicitly the dependence on the parameters. It is important to point out that all the estimates of the next sections will be uniform with respect to $\mu \in \mathcal{ U}$ and $\xi \in B(\xi_0,\sigma)$.  This will allow us to choose freely the values of $\mu$ and $\xi$ in Section \ref{Sec Par}. 
Consistently, the notation $O(f(x,\eps,\alpha,\beta,\de))$ and $o(f(x,\eps,\alpha,\beta,\de))$ will be used for quantities depending on $\eps, \xi, \mu$ (and the parameters $\alpha,\beta,\delta$ of Lemma \ref{Lemma parameters}) and satisfying respectively 
$$
|O(f(x,\eps,\mu,\xi \alpha,\beta,\de))|\le C f(x,\eps,\mu,\xi,\alpha,\beta,\de)) \quad \text{ and } \quad  \frac{o(f(x,\eps,\mu,\xi\alpha,\beta,\de))}{f(x,\eps,\mu,\xi,\alpha,\beta,\de)} \to 0,
$$
as $\eps \to 0$, uniformly for $\mu \in \mathcal{ U}$ and $\xi \in B(\xi_0,\sigma)$.

\section{The estimate of the error term}\label{Sec R}
In this section we give estimates for the error term $R$ defined in \eqref{DefR}
$$R= R_{\eps,\mu,\xi} := \Delta \omega_{\eps,\mu,\xi} + \lambda f_\eps(\omega_{\eps,\mu,\xi }) 
.$$

It will be convenient to split $\Om$ into four different regions:
\begin{equation}
\Om= B(\xi,\rho_0)\cup \Big( B(\xi,\rho_1)\setminus B(\xi,\rho_0) \Big) \cup\Big(B(\xi,\rho_2)\setminus B(\xi,\rho_1)\Big)\cup\Big(\Om\setminus B(\xi,\rho_2)\Big),
\end{equation}
where  $\rho_0= \rho_0(\eps,\mu,\xi)$, $\rho_1=\rho_1(\eps,\mu,\xi)$, $\rho_2=\rho_2(\eps,\mu,\xi)$, are defined by 
\begin{equation}\label{Radii}
\rho_{0}= \delta e^{\frac{\eps}{\alpha}}, \quad  \rho_{1}=e^{-\frac{u_0(\xi)}{2\alpha}} \quad \text{ and } \quad \rho_{2} = e^{-\frac{\eps}{\alpha}}. 
\end{equation} 
Note that  
$$
\delta\ll \rho_0\ll  \rho_1\ll\rho_2\ll 1, \qquad \text{ as }\eps \to 0,
$$ 
by \eqref{alpha} and \eqref{delta}.  
\medskip
Roughly speaking, we have to split the error into four parts:
 in $B(\xi,\rho_0)$ we have $  
\lambda f_\eps(\omega_{\eps}) = \alpha e^{U_{\eps}}(1+o(1))$ (see \eqref{motivation3}) and we can use a blow-up argument to get a uniform weighted estimate on $R$.  This estimate does not hold anymore in the set $\Om\setminus B(\xi,\rho_0)$, which we further split into three parts: the region $\Om\setminus B(\xi,\rho_2)$, where $\alpha G_\xi=O(\eps)$ and a uniform estimate on $R$ can be obtained via a Taylor expansion of $f_\eps(\omega_\eps)$ (using that $\omega_\eps = -V_\eps +8\pi \alpha G_{\xi}+o(\alpha^2)$),   and  the  two  annuli $B(\xi,\rho_1)\setminus B(\xi ,\rho_0)$ and $B(\xi,\rho_2)\setminus B(\xi ,\rho_1)$, where we give quite delicate  integral estimates. The last two regions are treated separately since $\omega_{\eps}\ge c_0>0$ in $B(\xi,\rho_1)\setminus B(\xi ,\rho_0)$, while $\omega_\eps$ changes sign in $B(\xi,\rho_2)\setminus B(\xi ,\rho_1)$ (cfr. Lemma \ref{positivity} and Lemma  \ref{negative}).

\subsection{A uniform expansion in \texorpdfstring{$B(\xi,\rho_{1})$}{B(xi,r1)}}
In this section we give a more precise version of the expasions in \eqref{motivation1}-\eqref{motivation2}. 

\begin{lemma}\label{(1+x)}
For any $\eps\in (0,1)$ and $x\ge -1$, we  have 
$$
|(1+x)^{1+\eps}  - 1- (1+\eps)x |\le \eps x^2. 
$$
\end{lemma} 
\begin{proof}
According to Bernoulli's inequality we have 
\begin{equation}\label{ber1}
(1+x)^{\eps} \le 1 +\eps x
\end{equation}
and 
\begin{equation}\label{ber2}
(1+x)^{1+\eps} \ge 1+ (1+\eps) x.
\end{equation}
Since $x\ge -1$, thanks to \eqref{ber1} we have that 
\begin{equation}\label{ber3}
(1+x)^{1+\eps} \le (1+x)(1+\eps x) = 1+ (1+\eps)x+ \eps x^2.
\end{equation}
Then, the conclusion follows from \eqref{ber2} and \eqref{ber3}.
\end{proof}

\begin{lemma}\label{positivity} Set $\displaystyle{c_0:= \frac{1}{2}\inf_{\xi \in B(\xi_0,\sigma)}u_0(\xi)}$. For $x\in B(\xi,\rho_1)$, we have that 
\begin{equation}\label{inB1}
 \beta + \alpha \bar U\Big( \frac{x-\xi}{\delta}\Big) \ge c_0,
\end{equation}
for sufficiently small $\eps$. In particular, we have 
\begin{equation}\label{inB1bis}
 c_0\le \omega_{\eps} \le \beta(1+o(1)).
\end{equation}
\end{lemma}
\begin{proof}
The definitons of $\ov U$ and $\rho_1$ (see \eqref{Ubar} and \eqref{Radii}), and \eqref{beta}-\eqref{delta} give
\[\begin{split}
\beta + \alpha \bar U \Big(\frac{x-\xi}{\delta}\Big)& \ge \beta + \alpha \bar U \Big(\frac{\rho_1}{\delta}\Big)\\
&= \beta - 4 \alpha \log \frac{\rho_1}{\delta} +o(1)\\
&= u_0(\xi)+o(1),
\end{split}
\]
which implies \eqref{inB1} for sufficiently small $\eps$. To get \eqref{inB1bis}, it is sufficient to apply Lemma \ref{expapproxsol} and Remark \ref{Rem beta}.  
\end{proof}

\begin{lemma}\label{expexp}
For $x\in B(\xi,\rho_1)$, we have 
\[
\omega_{\eps}^2(x) + \omega_{\eps}^{1+\eps}(x) = \beta^2+\beta^{1+\eps} + \bar U\bra{\frac{x-\xi}{\de}}+\alpha^2  \bar U^2\bra{\frac{x-\xi}{\de}}+O\bra{\eps \alpha^3 \bra{1+\bar U^2\Big(\frac{x-\xi}{\de}\Big)}}\hspace{-0.05cm}.
\]
\end{lemma}  
\begin{proof}
Set $y= \frac{x-\xi}{\de}\in B(0,\frac{\rho_1}{\de})$. Noting that $\ov U(y) = O(\alpha^{-2})$ and using  Lemma  \ref{expapproxsol}, we get  
\[\begin{split}
\omega_{\eps}^2(x)  = \omega_{\eps}^2(\xi+ \de y)  & = \bra{ \beta + \alpha \bar U (y) + O(\rho_{1}) }^2\\
& = \beta^2 + 2\alpha \beta \bar U(y) + \alpha^2 \bar U(y)^2 + O(\beta \rho_{1}).
\end{split}\]
Similarly, since  Lemma \ref{positivity} gives $\frac{\alpha}{\beta} \bar U(y) \ge -1 +\frac{c_0}{\beta} \ge -1$, by Lemma \ref{(1+x)} we infer 
\[\begin{split}
|\omega_{\eps}|^{1+\eps} (x) 
& =\beta^{1+\eps} \left(1 + \frac{\alpha}{\beta} \bar U(y) + O(\alpha \rho_{1}) \right)^{1+\eps} \\
& = \beta^{1+\eps}\left( 1 + (1+\eps) \Big( \frac{\alpha}{\beta} \bar U(y) +O(\alpha \rho_{1})\Big) + O\bra{\eps \bra{ \frac{\alpha}{\beta}  \bar U(y) +O(\alpha \rho_{1})}^2 } \right) \\
& = \beta^{1+\eps} + (1+\eps)\alpha \beta^{\eps} \bar U(y)  +O(\eps \alpha^3 (1+\bar U^2(y))).
\end{split}
\]
Then the conclusion follows from the second equation in \eqref{parameters}. 
\end{proof}

\subsection{Expansions in \texorpdfstring{$B(\xi,\rho_{0})$}{B(xi,r0)}}
Let us now restrict our attention to the smaller ball $B(\xi,\rho_0)$. This allows to control the term $\alpha^2 \bar U^2$ appearing in the expansion of Lemma \ref{expexp}. Indeed, since $|\bar U (y)| = -4\log |y|+O(1) $ as $|y|\to +\infty$, we have  that 
\begin{equation}\label{inBrho0}
\bar U\bra{\frac{x-\xi}{\delta}} = O\bra{\frac{\eps}{\alpha}} \quad \text{ and } \quad  \alpha^2 \bar U^2\bra{\frac{x-\xi}{\delta}} =O(\eps^2) \quad \text{ for  } x\in B(\xi ,\rho_0).  
\end{equation}

\begin{lemma}\label{RinBrho0} For  $x\in B(\xi ,\rho_{0})$, we have 
$$
R(x) = \alpha^3 e^{U_{\eps}(x)} \left(2 \bar U \Big(\frac{x-\xi}{\de}\Big)+ \bar U^2\Big(\frac{x-\xi}{\de}\Big)\right) + \alpha^4 e^{U_{\eps}(x)}O\left(1+\bar U^4\Big(\frac{x-\xi}{\de}\Big)\right).
$$
\end{lemma}
\begin{proof}
Set $y= \frac{x-\xi}{\de}$. First  by Lemma \ref{expapproxsol}, Lemma \ref{expexp}, and \eqref{parameters}-\eqref{scaling}, we get that 
\[
\begin{split}
\lambda f_\eps(\omega_{\eps}(x)) & = \lambda \beta \left(1+\frac{\alpha}{\beta}\bar U(y) + O(\alpha \rho_{1})\right)e^{\omega_{\eps}^2(x) + \omega_{\eps}^{1+\eps}(x) } \\
& = \frac{\alpha}{\delta^2} \left(1+ 2\alpha^2 \bar U(y) + O(\alpha^3(1 +|\bar U(y)|))\right)e^{\bar U(y)+\alpha^2 \bar U^2(y)+O(\eps \alpha^3 (1+\bar U^2(y)))}\\
& = \alpha e^{U_{\eps}(x)}\left(1+ 2\alpha^2 \bar U(y) + O(\alpha^3(1 +|\bar U(y)|))\right)e^{\alpha^2 \bar U^2(y)+O(\eps \alpha^3 (1+\bar U^2(y)))}.
\end{split}
\]
Now, by \eqref{inBrho0}, we can expand the last exponential term,  and find 
\[
\begin{split}
e^{\alpha^2 \bar U^2(y)+O(\eps \alpha^3 (1+\bar U^2(y)))}  & = 1 + \alpha^2 \bar U^2(y)+O(\eps \alpha^3 (1+\bar U^2(y))) + O(\alpha^4 (1+\bar U^4(y)))\\
& =  1 + \alpha^2 \bar U^2(y)+O(\eps \alpha^3 (1+\bar U^4(y))).
\end{split}
\] 
We can so conclude that
\begin{equation}\label{finBrho0}
\lambda f_\eps(\omega_{\eps}(x)) =  \alpha e^{U_\eps(x) } + \alpha^3e^{U_\eps(x) } \left( 2\bar U(y) +\bar U(y)^2\right) + \alpha^4e^{U_\eps(x)}O(1+\bar U^4(y)).
\end{equation} 
Moreover, by \eqref{0}-\eqref{eqzeps}, and Lemmas \ref{Lemma veps}-\ref{Lemma weps & zeps} we have 
\begin{equation}\label{LapinBrho0}
\Delta \omega_{\eps}  = -\alpha e^{U_{\eps}} + O(1)
= -\alpha e^{U_{\eps}} \left(1 + O(\alpha)e^{-U_{\eps}}\right) = -\alpha e^{U_{\eps}} (1+ o(\alpha^3)),
\end{equation}
where in the last equality we used that 
\[
e^{-U_{\eps}(x)} = \frac{(\de^2 \mu^2 +|x-\xi|^2)^2}{8\de^2\mu^2} = O( \de^2 e^{\frac{4\eps}{\alpha}}) = o(\alpha^3),
\]
for $x\in B(\xi,\rho_0)$. Thanks to  \eqref{finBrho0} and \eqref{LapinBrho0}, we  conclude that 
$$
R(x)= \alpha^3 e^{U_\eps(x)} \left(2\bar U (y)+ \bar U^2(y)\right) + \alpha^4 e^{U_{\eps}(x)}O(1+\bar  U^4(y)).
$$
\end{proof}

As an immediate consequence of the previous lemma we obtain the estimate:

\begin{cor}\label{EstRinB0}
We have that 
$$
R=O\left(\alpha^3e^{U_{\eps}}\left((1+\bar U^4\Big(\frac{\cdot - \xi }{\de}\Big)\right)\right)
$$
in $ B(\xi,\rho_0)$. 
\end{cor}

\subsection{Estimates on \texorpdfstring{$B(\xi,\rho_1)\setminus B(\xi,\rho_0)$}{B(0,r1)-B(0,r0)}}
\medskip
In this region, it is diffcult to provide pointwise estimates of $R$ because the term $\alpha^{2} \bar U^2$ appearing in the expansion of Lemma \ref{expexp} becomes very large. Then, we will look for integral estimates. Specifically we will show that $R$ is (very) small in $L^{p}(B(\xi,\rho_1)\setminus B(\xi,\rho_0))$, for a suitable choice of $p=p(\alpha)>1$, such that $p \to 1$ as $\eps \to 0$, uniformly with respect to $\xi \in B(\xi_0,\sigma )$, $\mu \in \mathcal U$. 

\begin{lemma}\label{finBrho1}
There exists $c_1>0$ such that 
\[\begin{split}
0\le \lambda f_\eps(\omega_{\eps}) \le \alpha e^{U_{\eps} + \alpha^2 (1+ c_1\eps\alpha) \bar U^2(\frac{\cdot - \xi }{\de})},
\end{split}
\]
in $B(\xi ,\rho_1)\setminus B(\xi,\rho_0)$. 
\end{lemma}
\begin{proof}
Since  $0\le \omega_\eps \le \beta$ in $B(\xi,\rho_1)\setminus B(\xi,\rho_0)$, from Lemma \ref{expexp} and \eqref{parameters} we get
\[
\begin{split}
\lambda f_\eps(\omega_{\eps})  & \le \lambda \beta e^{\beta^2+\beta^{1+\eps}+\bar U(\frac{\cdot - \xi }{\de})+\alpha^2 \bar U^2(\frac{\cdot -\xi }{\de})(1+O(\eps \alpha))} \\ 
& = \frac{\alpha}{\delta^2} e^{\bar U (\frac{\cdot - \xi }{\de}) +\alpha^2 \bar U^2(\frac{\cdot - \xi }{\de}) (1+ O(\eps \alpha))}\\
&  = \alpha e^{U_{\eps} +\alpha^2 \bar U^2(\frac{\cdot - \xi }{\de}) (1+ O(\eps \alpha))}.
\end{split}
\]
\end{proof}

For $c_1$  as in Lemma \ref{finBrho1}, let us consider the function 
\begin{equation}\label{Gamma}
\Gamma_\eps(x):=  e^{\bar U_\eps(x)
+\alpha^2 \bar U(\frac{x-\xi }{\de})^2(1+ c_1\eps \alpha )}.
\end{equation}

\begin{lemma}\label{intGamma}
Set $p:=1+\alpha^2$. There exists $c_2>0$ such that
$$
\|\Gamma_\eps\|_{L^p(B(\xi ,\rho_1)\setminus B(\xi,\rho_0))} = O\left( \alpha^{-1} e^{-\frac{c_2}{\sqrt{\alpha}}} \right). 
$$
\end{lemma}
\begin{proof}
First of all, we observe that for $q\in (\frac{1}{2},+\infty)$, $R>0$, one has 
\begin{equation}\label{bound int}
\int_{\R^2 \setminus B(0,R)} e^{q \ov U} dy \le  \int_{\R^2 \setminus B(0,R)} \frac{(8\mu^2)^q}{|y|^{4q}} dy = \frac{\pi (8\mu^2)^q}{(2q-1)R^{4q-2}}.  
\end{equation}
For $x\in B(\xi,\rho_1) \setminus B(\xi,\rho_0)$, set $y = \frac{x-\xi}{\de}\in B(0,\frac{\rho_1}{\delta})\setminus B(0,\frac{\rho_0}{\delta})$. Clearly we have 
\begin{equation}\label{Gamma1}
\|\Gamma_\eps\|_{L^p(B(\xi ,\rho_1)\setminus B(\xi,\rho_0) )} = \delta^{\frac{2-2p}{p}} \bra{\int_{B(0, \frac{\rho_1}{\de}) \setminus B(0,\frac{\rho_{0}}{\delta})} e^{p \bar U(y)
(1+\alpha^2 \bar U(y) (1+ c_1\eps \alpha ))}dy}^\frac{1}{p}.
\end{equation}
Set $\bar \rho = \delta e^{\frac{1}{\alpha^{\frac{3}{2}}}}$, so that $\rho_0 \ll  \bar \rho \ll \rho_1$. For $\frac{\rho_0}{\de}\le |y|\le \frac{\bar \rho}{\de}$, we have 
\[
p \bra{ 1+\alpha^2   \bar U(y) (1+\eps c_1 \alpha) } = 1+O(\sqrt{\alpha})\ge \frac{2}{3} .
\]
Then, for $\eps$ small enough, \eqref{bound int} yields
\begin{equation}\label{Gamma2}
\begin{split}
 \int_{B(0, \frac{\bar \rho}{\de}) \setminus B(0,\frac{\rho_{0}}{\delta})} e^{ p \bar U(y) \left(1
+\alpha^2 \bar U(y)(1+ c_1\eps \alpha )\right)}dy & 
\le  \int_{\R^2 \setminus B(0,{\frac{\rho_0}{\delta}})} e^{\frac{2}{3}   \bar U(y)} dy \\ 
&= O\left(\left(\frac{\rho_{0,\eps}}{\delta}\right)^{-\frac{2}{3}}\right)= O( e^{-\frac{2 \eps}{3\alpha}}).
\end{split} 
\end{equation}
For $\frac{\bar \rho}{\de}\le |y|\le \frac{\rho_1}{\de}$, by \eqref{beta} and  Lemma \ref{positivity}, we have 
\[\begin{split}
1+  \alpha^2\bar U(y)\bra{1+c_1\eps\alpha} & = 1+   \alpha (\beta+ \alpha \bar U(y)) \bra{1+c_1\eps\alpha}  - \alpha \beta \bra{1+c_1\eps \alpha}  \\ & \ge \frac{1}{2} + (c_0 +u_0(\xi))\alpha + o(\alpha) \\
& \ge \frac{1}{2}+c_0\alpha.
\end{split}
\] 
Hence, we get
\begin{equation}\label{Gamma3}
\begin{split}
 \int_{B(0, \frac{\rho_1}{\de}) \setminus B(0,e^{\alpha^{-\frac{3}{2}}})} e^{ p \bar U(y) \left(1
+\alpha^2 \bar U(y)(1+ c_1\eps \alpha )\right)}dy &  \le  \int_{\R^2 \setminus B(0,e^{\alpha^{-\frac{3}{2}}})} e^{ {p}(\frac{1}{2}+c_0 \alpha)\bar U(y) } dy \\ 
&= O\Big(\alpha^{-1} e^{-\frac{4c_0}{\sqrt{\alpha}}}\Big).
\end{split}
\end{equation}
Thus, by \eqref{Gamma1},\eqref{Gamma2},\eqref{Gamma3}, we obtain
$$
\|\Gamma_\eps\|_{L^p(B(\xi ,\rho_1)\setminus B(\xi,\rho_0) )}  = O\bra{ \delta^\frac{{2-2p}}{p}\alpha^{-\frac{1}{p}} e^{-\frac{4c_0}{p\sqrt{\alpha}}} }.
$$
Since  \eqref{alpha}-\eqref{delta} give
$$
\delta^{\frac{2-2p}{p}} = \delta^{-\frac{2\alpha^2}{1+\alpha^2}} = O(1),\qquad  \alpha^\frac{1}{p} = \alpha \alpha^{\frac{1-p}{p}}=\alpha(1+o(1)), \qquad  e^{-\frac{4c_0}{p\sqrt{\alpha}}} = O( e^{-\frac{4c_0}{\sqrt{\alpha}}}),$$
we get the conclusion.  
\end{proof}

\begin{lemma}\label{RinBrho1}
Let $p$ and  $c_2$ be as in Lemma \ref{intGamma}, then  
$$
\| R \|_{L^p(B(\xi,\rho_1)\setminus B(\xi,\rho_0))} = O( e^{-\frac{c_2}{\sqrt{\alpha}}}). 
$$
\end{lemma}
\begin{proof}
By Lemma \ref{finBrho1} and Lemma \ref{intGamma}  we get that
$$
\|\lambda f_\eps (\omega_{\eps})\|_{L^p(B(\xi,\rho_1)\setminus B(\xi,\rho_0))}  =O(e^{-\frac{c_2}{\sqrt{\alpha}}}).
$$
On the other hand, we have
\[
\Delta \omega_{\eps} (x)= - \alpha e^{U_{\eps}(y)} + O(1),
\]
so that
\[
\begin{split}
\|\Delta \omega_{\eps}\|_{L^p(B(\xi,\rho_1)\setminus B(\xi,\rho_0))} & \le  \alpha \|e^{U_{\eps}}\|_{L^p(B(\xi,\rho_1)\setminus B(\xi,\rho_0))}  +O(\rho_{1}^\frac{2}{p}) \\
& \le  \alpha \delta^\frac{2-2p}{p} \|e^{\bar U}\|_{L^p(\R^2\setminus B(0,\frac{\rho_0}{\de}))}  +O(\rho_{1}^\frac{2}{p}) \\
& =  O\bra{\frac{\alpha \delta^2}{\rho_{0}^2}} + O \bra{\rho_{1}^2} \\
& =o( e^{-\frac{c_2}{\sqrt{\alpha}}}).
\end{split}
\]
\end{proof}

\subsection{Estimates in \texorpdfstring{$B(\xi,\rho_2)\setminus B(\xi,\rho_1)$}{B(0,r2)-B(0,r1)}}
In  $B(\xi,\rho_2)\setminus B(\xi,\rho_1)$ we can only say that $\omega_{\eps}$ and $R$ are uniformly bounded.  Since $\rho_2$ is very small, we still get integral bounds for $R$.

\begin{lemma}\label{RinBrho2}
We have $\omega_{\eps}=O(1)$ and $R=O(1)$ in $\Omega\setminus B(\xi,\rho_1)$.  In particular, 
$$
\|R\|_{L^{2}(B(\xi,\rho_2)\setminus B(\xi,\rho_1))}= O(\rho_2)= O(e^{-\frac{\eps}{\alpha}}). 
$$
\end{lemma} 
\begin{proof} 
Let us recall that $\omega_\eps =  \alpha P U_\eps - V_\eps$ with $V_\eps=V_{\eps,\alpha,\xi}$ defined as in \eqref{Veps}. According to Lemma \ref{Lemma veps} and Lemma \ref{Lemma weps & zeps}, we have $V_\eps=O(1)$  in $\Omega$. Besides  Lemma \ref{3} gives
$$
\alpha PU_{\eps} = \alpha \log\left( \frac{1}{(\mu^2 \delta^2 + |x-\xi|^2)^2}\right) +O(\alpha) = O( \alpha \log \frac{1}{\rho_1}) +O(\alpha) = O(1),
$$
for $x\in \Omega \setminus B(\xi,\rho_1)$. Then, $\omega_{\eps}=O(1)$ and $f_\eps(\omega_{\eps})=O(1)$ in $\Omega \setminus B(\xi,\rho_1)$. Similarly
\[
\begin{split}
\Delta\omega_{\eps} & = -\alpha e^{U_{\eps}} + O(1) \\
& = -\frac{\alpha \delta^2 \mu^2}{(\delta^2\mu^2+|x-\xi|^2)^2} +O(1) \\
&= O(\delta^{2}\rho_{1}^{-4}) + O(1)= O(1).
\end{split}\]
Therefore $R=O(1)$. 
\end{proof}

\subsection{Estimates in \texorpdfstring{$\Omega\setminus B(\xi,\rho_2)$}{O-B(0,r2)}} 
In  $\Omega \setminus B(\xi,\rho_{2})$ we will use that ${\omega_{\eps}} \sim 8\pi \alpha G_\xi -V_\eps$. Our choice of $V_\eps$  will make $R$  uniformly small, namely of order $\alpha^3$.  Note further that the choice of $\rho_2$ gives $\alpha { G_\xi } = O(\eps)$ on $\Omega\setminus B(\xi,\rho_2)$. 

\begin{lemma}\label{LastRegion}
As $\eps\to 0$ we have 
$$
\|PU_\eps - 8\pi G_\xi \|_{C^1(\ov{ \Omega} \setminus B(\xi,\rho_2))} = O(\delta^2 \rho_{2}^{-3}). 
$$ 
\end{lemma}
\begin{proof}
By Lemma \ref{3} we have 
\[\begin{split}
PU_{\eps} & = \log \bra{ \frac{1}{(\de^2\mu^2 + |x-\xi |^2)^2}} + 8\pi H(x,\xi)+\psi_{\de,\mu,\xi} \\
& = -4 \log |x-\xi| + 8\pi H(x,\xi)  - 2 \log  \left( 1 + \frac{\de^2\mu^2}{|x-\xi|^2}\right)+\psi_{\de,\mu,\xi} \\
& = 8\pi G_\xi(x ) - 2 \log  \left( 1 + \frac{\de^2\mu^2}{|x-\xi|^2}\right)+\psi_{\de,\mu,\xi}
\end{split}\]
Since $\|\psi_{\de,\mu,\xi}\|_{C^1(\ov \Omega)} =O(\de^2)$ as $\eps\to 0$, it is sufficent to observe that 
$$\|\log  \left( 1 + \frac{\de^2\mu^2}{|\cdot-\xi |^2}\right)\|_{C^1(\ov \Omega\setminus B(\xi,\rho_2))}= O(\delta^2 \rho_2^{-3}).$$ 
\end{proof}

\begin{lemma}\label{negative}
There exists a constant $c>0$ such such that   
$$
\omega_{\eps} (x) \le -c \,d(x,\partial \Omega)< 0,  
$$
for any $x\in \Omega\setminus B(\xi,\rho_2)$, provided $\eps$ is sufficiently small. 
\end{lemma}
\begin{proof}
By Lemma \ref{Lemma veps}, Lemma \ref{Lemma weps & zeps} and \eqref{Veps} we have 
$$
V_\eps(x ) \ge c(1+O(\alpha)) d(x,\partial \Omega) \qquad \forall x\in \Omega,
$$
for some $c>0$. Then, Lemma \ref{LastRegion} implies that  
\begin{equation}\label{Boh}
 \omega_{\eps}(x) \le  - c (1+O(\alpha)) d(x,\partial \Omega)
\end{equation}
in a neighborhood of $\partial \Omega$.  By definiton of $\rho_2$, we have  that $P U_\eps = G_\xi + o(1) = O( \frac{\eps}{\alpha})$ in $\Omega \setminus B(\xi,\rho_2)$. Then, using again  Lemma \ref{Lemma veps} and Lemma \ref{Lemma weps & zeps}, we  get $\omega_\eps  = -u_0 + o(1)$  uniformly in $\Omega \setminus B(\xi,\rho_2) $. Since $u_0>0$ in $\Omega$, this toghether with \eqref{Boh} yields the conclusion.  
\end{proof}

\begin{lemma}\label{Routside}
In  $ \Omega\setminus B(\xi,\rho_2)$, we have $R = O(\alpha^3(1+{G_\xi^3}))$. In particular, $$\|R\|_{L^2(\Omega\setminus B(\xi,\rho_2))}=O(\alpha^3).$$
\end{lemma}
\begin{proof}
Since $v_\eps>0$ in $\Omega$,  $\omega_\eps<0$ in $\Omega\setminus B(\xi,\rho_2)$, and $f_\eps \in C^3((-\infty,0))$,   for any $x\in \Omega\setminus B(\xi,\rho_2)$ we can find $\theta(x)\in [0,1]$ such that 
\[
\begin{split}
f_\eps(\omega_\eps) &= f_\eps(-v_\eps +\alpha PU_\eps  -\alpha w_\eps - \alpha^2 z_\eps) \\
& =  f_\eps(-v_\eps) + f_\eps'(-v_\eps)(\alpha PU_\eps -\alpha w_\eps - \alpha^2 z_\eps) + \frac{1}{2} f_\eps''(-v_\eps) (\alpha PU_{\eps}  - \alpha w_\eps - \alpha^2 z_\eps )^2  \\
&\quad + {\frac{1}{6}}f'''( -v_\eps + \theta (\alpha PU_{\eps} - \alpha w_\eps - \alpha^2 z_\eps)) (\alpha PU_{\eps} - \alpha w_\eps - \alpha^2 z_\eps)^3\\
\end{split}
\]
According to Lemma \ref{Lemma weps & zeps} and  Lemma \ref{LastRegion}, we have \[\begin{split}
|z_\eps|+|w_\eps|=O(G_\xi) \quad \text{ and } \quad \alpha PU_{\eps}  = 8\pi\alpha G_\xi(1  + o(\alpha^3)). 
\end{split}\]
Thus we get
\[
\begin{split}
f_\eps(\omega_\eps) & = -f_\eps(v_\eps) + \alpha f_\eps'(v_\eps)(8\pi G_\xi- w_\eps) +  \alpha^2 \left(\frac{1}{2}f''(-v_\eps)(8\pi G_\xi-w_\eps)^2 - f'(v_\eps)z_\eps   \right) \\
& \qquad +  O(\alpha^3(1+ G_\xi^3))+  O(\alpha^3 |f'''( -v_\eps + \theta (\alpha PU_{\delta,\mu} - \alpha w_\eps - \alpha^2 z_\eps))|{G_\xi^3} ).
\end{split}\]
A direct computation shows the existence of a constant $C>0$ such that 
$$
|f'''_\eps(t)| \le C (|t|^{\eps-1} + t^4) e^{t^2+|t|^{1+\eps}}\quad \forall t\neq 0.
$$
Since $-v_\eps + \theta (\alpha PU_{\eps} - \alpha w_\eps -\alpha^2 z_\eps)=O(1)$ uniformly in ${\Omega\setminus B(\xi,\rho_2)}$, and since Lemma \ref{LastRegion} implies
$-v_\eps + \theta (\alpha PU_{\eps} + \alpha w_\eps + \alpha^2 z_\eps) \le -c d(\cdot,\partial \Omega)$ in a neighborhood of $\partial \Omega$, we get 
$$
|f'''( -v_\eps + \theta (\alpha PU_{\delta,\mu} - \alpha w_\eps - \alpha^2 z_\eps))|= O(1+ d(\cdot,\partial\Omega)^{\eps-1}). 
$$
Since $G_\xi=O(d(\cdot,\partial\Omega))$ near $\partial \Omega$, we deduce that 
\[
\begin{split}
f_\eps(\omega_{\eps})& = -f_\eps(v_\eps) + \alpha f_\eps'(v_\eps)(8\pi G_\xi -w_\eps) +  \alpha^2 \left(\frac{1}{2}f''(-v_\eps)(8\pi G_\xi -w_\eps)^2 - f'_\eps (v_\eps)z_\eps   \right) \\
& \qquad + O(\alpha^3(1+{G_\xi^3})).
\end{split}
\]
Since by construction we have
$
\Delta \omega_{\eps} = - \alpha e^{U_{\eps}}  -\Delta v_\eps  - \alpha \Delta w_\eps - \alpha^2 \Delta z_\eps $,
with $v_\eps$, $w_\eps$, $z_\eps$ solving \eqref{problem feps} and  \eqref{eqweps}-\eqref{eqzeps}, we conclude that
\[\begin{split}
R & = - \alpha e^{U_{\eps}}+ O(\alpha^3(1+{G_\xi^3})) \\
&=O(\delta^2 \rho_2^{-4})+ O(\alpha^3(1+{G_\xi^3}))\\
&= O(\alpha^3(1+{G_\xi^3})).
\end{split}\]
\end{proof}

\subsection{The final estimate of the error in a mixed norm}
We can summarize the estimates of the previous sections as follows: 

In $B(\xi,\rho_0)$, Corollary \ref{EstRinB0} gives  $|R|\le \alpha^3 j_\eps,$ where 
\begin{equation}\label{Def jeps}
j_\eps(x) := e^{U_{\eps}(x)} \bra{1+ |\bar U\Big(\frac{x-\xi}{\delta}\Big)|^4}. 
\end{equation}

In $B(\xi,\rho_1) \setminus B(\xi,\rho_0)$, Lemma \ref{RinBrho1} shows that the norm of $R$  in $L^{1+\alpha^2}$ is exponentially small in $\alpha$.

 Finally, in $\Omega\setminus B(\xi,\rho_1)$, Lemma \ref{RinBrho2} and Lemma \ref{Routside} give $L^2$ estimates on $R$. This suggests to introduce the norm 
\begin{equation}\label{norm}
\| f\|_{\eps}:= \| j_\eps^{-1} f \|_{L^\infty (B(\xi,\rho_0))} +  \frac{1}{\alpha^2}\| f\|_{L^{1+\alpha^2}(B(\xi,\rho_1) \setminus B(\xi,\rho_0))} + \|f\|_{L^2(\Omega\setminus B(\xi,\rho_1))}.
\end{equation}
The coefficient $\frac{1}{\alpha^2}$ is chosen in order to match the norm of $(-\Delta)^{-1}$ as a linear operator from $L^{1+\alpha^2}( B(\xi,\rho_1) \setminus B(\xi,\rho_0) )$ into $L^\infty( B(\xi,\rho_1) \setminus B(\xi,\rho_0) )$ (see Corollary \ref{CorStamp2}).

According to the estimates above we have:
 
\begin{prop}\label{EstR}
There exists $D_1>0$, $\eps_0>0$ such that 
$$
\|R\|_{\eps} \le D_1\alpha^3, 
$$
for any $\eps \in (0,\eps_0)$, $\mu\in \mathcal U$, $\xi \in B(\xi_0,\sigma)$. 
\end{prop}

We conclude this section by stating some simple properties of the norm $\|\cdot\|_\eps$ and the weight $j_\eps$.

\begin{lemma}\label{L1}
There exists a constant $C>0$ such that 
$$
 \| \cdot  \|_{L^1(\Omega)} \le C \|\cdot\|_{\eps} 
$$
for any $\eps >0$, $\mu\in \mathcal U$, $\xi \in B(\xi_0,\sigma)$. 
\end{lemma}
\begin{proof}
Let $f:\Omega \ra \R$ be a Lebesgue measurable function. Then 
\[\begin{split}
\|f\|_{L^1(B(\xi,\rho_0))} \le \|f\|_{\eps} \int_{B(\xi,\rho_0)} j_\eps  dx  = \|f\|_{\eps}\int_{B(0,\frac{\rho_0}{\de})} e^{\bar U}(1+\bar U^4) dy   \le C \|f\|_{\eps}.
\end{split}\]
By H\"older's inequality
\[
\|f\|_{L^1(B(\xi,\rho_1)\setminus B(\xi,\rho_0))  } \le 
\|f\|_{L^{1+\alpha^2}(B(\xi,\rho_1)\setminus B(\xi,\rho_0)) } \rho_1^\frac{2\alpha^2}{1+\alpha^2} \le C \|f\|_{\eps},
\]
and
\[
\|f\|_{L^1(\Omega\setminus B(\xi,\rho_1)) } \le  \|f\|_{L^2(\Omega\setminus B(\xi,\rho_1)) }|\Omega\setminus B(\xi,\rho_1)|^\frac{1}{2} \le C \|f\|_{\eps}. 
\]
Hence, the conclusion follows. 
\end{proof}

\begin{lemma}\label{barrier}
For any $\eps>0$ let $\rho_{\eps}$, $\sigma_\eps$ be such that $\rho_2 \le \sigma_\eps \le \sigma $ and  $\delta \ll \rho_\eps \le \rho_0$ as $\eps \to 0$. Let  $\ph_\eps$  of be the solution to
$$
\begin{cases}
-\Delta \ph_\eps  = j_\eps  & \text{ in }  B(\xi,\sigma_\eps)\setminus B(\xi,\rho_{\eps}),\\
\ph_\eps = 0 & \text{ on }\partial B(\xi,\sigma_\eps)\setminus B(\xi,\rho_{\eps}). 
\end{cases}
$$
As $\eps \to 0$, we have 
$$
\|\ph_\eps\|_{L^\infty(B(\xi,\sigma_\eps) \setminus B(\xi,\rho_{\eps})) }=o(1).
$$
\end{lemma}
\begin{proof}
Let us first note that there exists a constant $c>0$, such that 
$$
{\de^2}j_{\eps}(\xi + \delta \, \cdot ) = e^{\bar U} (1+\bar U^4)  =   \frac{8 \mu^2 }{(\mu^2+|\cdot|^2)^{2}} \left(1+ \log^4\left( \frac{8\mu^2}{(\mu^2 + |\cdot|^2)^2} \right) \right) \le c \frac{\mu}{(\mu^2+|\cdot|^2)^\frac{3}{2}} 
$$
in $\R^2$.  Then, by the maximum principle,  we have 
\begin{equation}\label{MaxPrin}
|\ph_{\eps}|\le c \psi \bra{ \frac{\cdot -\xi }{\delta}} \quad \text{in } B(\xi,\sigma_\eps)\setminus B(\xi,\rho_{\eps}),
\end{equation}
where $\psi $ satisfies 
$$
\begin{cases}
-\Delta \psi = \frac{\mu}{(\mu^2+|\cdot|^2)^\frac{3}{2}}  & \text{ in }A_\eps:= B(0,\frac{\sigma_\eps} {\de}) \setminus B(0,\frac{\rho_\eps}{\de}),\\
\psi = 0 & \text{ on }\partial A_\eps. 
\end{cases}
$$
Since the function $W := -\log (\mu +\sqrt{|\cdot|^2+\mu^2})$ satisfies $-\Delta W = \frac{\mu}{(\mu^2+|\cdot|^2)^\frac{3}{2}}$, we have
$$
\psi = a + b \log |\cdot| + W,
$$
for suitable constants $a,b \in \R$. Denoting $R_1 = \frac{\rho_{\eps}}{\de}$ and $R_2 = \frac{\sigma_\eps}{\de}$ one can  verify that 
$$
a = \frac{W(R_2)\log R_1-W(R_1)\log R_2}{\log R_2-\log R_1}  \qquad \text{ and }  \qquad  b = \frac{W(R_1)-W(R_2)}{\log R_2- \log R_1}.
$$
Since 
$$
|W + \log |\cdot|| \le \frac{C\mu}{|\cdot|} = O\bra{\frac{1}{R_1}},
$$ 
uniformly in $\ov{A_\eps}$, one has $a = O\bra{\frac{\log R_2}{R_1 (\log R_2-\log R_1)}}$ and $b = 1+O\bra{\frac{1}{R_1 (\log R_2-\log R_1)}}$. Then 
\[
\begin{split}
\psi & = a + (b-1) \log |\cdot| + O(\frac{1}{R_1}) \\
& =  O\left( \frac{1}{R_1}  \frac{\log R_2}{\log R_2 -\log R_1}\right) + O\bra{\frac{1}{R_1}}  \\
& = O\left(\frac{1}{R_1}  \frac{1}{1 -\frac{\log R_1}{\log R_2}}\right) +  O\bra{\frac{1}{R_1}}.
\end{split}
\]
Since 
$$
\frac{\log R_1}{\log R_2} = \frac{\log \frac{\rho_\eps}{\de}}{\log \sigma_\eps  - \log \delta } \le  \frac{\log \frac{\rho_0}{\de}}{\log \rho_2 - \log \delta } =O(\alpha),
$$
we conclude that $\psi_{\mu} = O(\frac{1}{R_1})=o(1)$, uniformly in $A_\eps$. Then, the conclusion follows  by  \eqref{MaxPrin}. 
\end{proof}

\section{The Linear Theory}\label{Sec L}
Let us consider the linear operator 
\[
L\ph = \ph - (-\Delta)^{-1} (\lambda f'_\eps(\omega_{\eps})   \ph)
\]
introduced in \eqref{DefL}. In this section we give a priori estimates for the operator $L$ and we prove its invertibility on a suitable subspace of $H^1_0(\Omega)$.

\begin{lemma}\label{Est L} The following expansions hold: 
\begin{enumerate}
\item $\lambda f_\eps'(\omega_\eps) =  e^{U_{\eps}}(1+O(\eps^2)) $  in $B(\xi,\rho_0)$.
\item $\lambda f_\eps'(\omega_\eps) = O(\Gamma_\eps)$ in $B(\xi,\rho_1)$, with $\Gamma_\eps$ as in \eqref{Gamma}.
\item $\lambda f_{\eps}'(\omega_\eps) = O(1)$ in $\Omega \setminus B(\xi,\rho_1)$. 
\item ${ \|\lambda f_{\eps}'(\omega_\eps) \chi_{B(\xi,\rho_1)}- e^{U_\eps}\|_{\eps} = o(1)} $ as $\eps \to 0$.
\end{enumerate}
\end{lemma}
\begin{proof}
For $x\in B(\xi,\rho_0)$, using  \eqref{parameters}-\eqref{scaling}, Lemma \ref{expexp}, \eqref{cru1}, and \eqref{inBrho0}, we have that 
\[
\begin{split}
\lambda f'_\eps(\omega_{\eps}) & = \lambda (1+2\omega_{\eps}^2+ (1+\eps)\omega_{\eps}^{1+\eps}) e^{\omega_{\eps}^2 + \omega_{\eps}^{1+\eps}} \\ 
&= \lambda \beta^2(2+O(\alpha)) e^{\beta^2+\beta^{1+\eps}+ \bar U(\frac{\cdot - \xi }{\delta}) + O(\eps^2)} \\
& =  e^{U_{\eps}}(1+O(\eps^2)).  
\end{split}
\]
For $x\in B(\xi,\rho_1)$, using Remark \ref{Rem beta}, Lemma \ref{expexp}  we have\[
\begin{split}
\lambda f'_\eps(\omega_{\eps}) & = \lambda (1+2\omega_{\eps}^2+ (1+\eps)\omega_{\eps}^{1+\eps}) e^{\omega_{\eps}^2 + \omega_{\eps}^{1+\eps}} \\ 
&= \lambda \beta^2(2+O(\alpha)) e^{\beta^2+\beta^{1+\eps}+ \bar U(\frac{\cdot}{\delta}) + \bar U(\frac{\cdot-\xi}{\delta})^2(1+O(\eps \alpha))} \\
& =  O\left(  \Gamma_\eps \right).
\end{split}
\]
{ Claim \emph{3}  follows directly from Lemma \ref{RinBrho2}. Finally, claim 4 follows by claims \emph{1} and \emph{2}, using also Lemma \ref{intGamma} and the estimates 
$$
\|e^{U_\eps}\|_{L^{1+\alpha}(B(\xi,\rho_1)\setminus B(\xi,\rho_0))} =o(1), \quad   \|e^{U_\eps}\|_{L^2(\Omega\setminus B(\xi,\rho_1))} =o(1).
$$
} 
\end{proof}

According to Lemma \ref{Est L},  for $|x-\xi|\le \rho_0$,  $L$ approaches the operator 
$L_0\ph := \ph - (-\Delta)^{-1} (e^{U_{\eps}}\ph)$. Note that 
\[
\begin{split}
L_0 \ph  = 0 \quad \text{ in } \Omega \qquad &\Longleftrightarrow \qquad -\Delta \ph = e^{U_{\eps}}\ph \quad \text{ in }\Omega \\ 
& \Longleftrightarrow  \qquad  -\Delta \Phi = e^{\bar U}\Phi \quad \text{ in }\frac{\Omega-\xi} {\delta}, \text{ where } \Phi = \ph(\xi + \delta\, \cdot ). 
\end{split}
\]
Let us recall the following known fact about $L_0$ (see for example \cite{BP}).

\begin{prop}\label{5}
All  bounded weak solutions of the problem
\begin{equation}
-\Delta \Phi= e^{\bar U }\Phi\quad\hbox{ in }\R^2
\end{equation}
have the form
$$
\Phi = c_0 Z_0 + c_1 Z_1 + c_2 Z_2,
$$ 
where $c_0,c_1,c_2 \in \R$ and  
$$ 
Z_0(y):= \frac{\mu^2 - |y|^2}{\mu^2+ |y|^2}, \qquad  Z_1(y):= \frac{2\mu y_1}{\mu^2+ |y|^2}, \qquad  Z_2(y):= \frac{2\mu y_2}{\mu^2+ |y|^2}.
$$
\end{prop}

\begin{rem} The functions $Z_0,Z_1,Z_2$ are orthogonal in $D^{1,2}(\R^2)$, that is 
\begin{equation}\label{orth}
\int_{\R^2} \nabla Z_i \cdot \nabla Z_j dy =  \int_{\R^2} e^{\bar U}Z_i Z_j dy = \frac{8}{3}\pi \delta_{i,j}. 
\end{equation}
\end{rem}

In the following we denote 
$$
Z_{i,\eps}(x) :=  Z_{i}\bra{ \frac{x-\xi}{\delta}  }\quad \text{ and } \quad P Z_{i,\eps} = (-\Delta)^{-1} Z_{i,\eps},  \qquad i =0,1,2.  
$$
\begin{lemma}\label{7} 
It holds true that
$$PZ_{0,\eps}=Z_{0,\eps}+1+O(\de^2)\ \hbox{and}\ 
 PZ_{i,\eps} = Z_{i,\eps} + O(\de),\ i=1,2,
$$
uniformly with respect to $\mu \in \mathcal U $, $\xi \in B(\xi_0,\sigma)$.  
\end{lemma}

\begin{proof} See for example Appendix A in \cite{egp}.
\end{proof}

Lemma \ref{7} shows the smallness of $PZ_{i,\eps} - Z_{i,\eps}$ for $i=1,2$, but not for $i=0$.  For this reason, in many cases it is convenient to replace $P Z_{0,\eps}$ with   the funtion 
\begin{equation}\label{Def Zbis}
\wt Z_{\eps}:= \begin{cases}
Z_{0,\eps} & \text{ if } |x-\xi|\le \rho_{0},\\
 Z_{0,\eps}(\rho_{0}) (\frac{\log \rho_{1}-\log |x-\xi|}{\log \rho_{1} - \log \rho_{0}}) & \text{ if }  \rho_{0} \le |x-\xi|\le \rho_{1}, \\ 
0 & \text{ if } |x-\xi|\ge \rho_{1}.
\end{cases}
\end{equation}

\begin{lemma}\label{Zbis}
The function $\wt Z_\eps$ satisfies the following properties:
\begin{itemize}
\item $\wt Z_\eps\in H^1_0(\Omega)$ and $|\wt Z_{\eps}|\le 1$ in $\Omega$.  
\item $\|\nabla(\wt Z_{\eps} -Z_{0,\eps})\|_{L^2(\Omega)} \to 0$, uniformly for $\mu \in \mathcal U$ and $\xi \in B(\xi_0,\sigma)$.
\end{itemize}
\end{lemma}
\begin{proof}
The first property follows trivially from the definition. Moreover we have
\[\begin{split}
\|\nabla (\wt Z_{\eps} -Z_{0,\eps})\|_{L^2(\Omega)}^2 & \le  \frac{Z_{0,\eps}(\rho_{0})^2}{(\log \rho_{1} -\log \rho_{0})^2} \hspace{-0.08cm} \int_{B(\xi,\rho_{1}) \setminus B(\xi,\rho_{0})} \hspace{-0.07cm}\frac{1}{|x-\xi|^2} dx   + \|\nabla Z_{0,\eps}\|_{L^2(\Omega\setminus B(\xi,\rho_{0}))}^2  \\
& \le  \frac{2\pi Z_{0,\eps}(\rho_{0})^2}{\log \rho_{1} -\log \rho_{0}} + \|\nabla Z_{0}\|_{L^2(\R^2 \setminus B(0,\frac{\rho_{0}}{\delta}))}^2\\
& = O(\alpha^2)+ O(e^{-\frac{\eps}{\alpha}}) \to 0,
\end{split}
\]
as $\eps \to 0$. 
\end{proof}

We will denote by $K_{\eps}$ the subspace of $H^1_0(\Omega)$ spanned by $PZ_{i,\eps}$, $i=0,1,2$ and by  $K_{\eps}^{\perp}$  the subspaces of $H^1_0(\Omega)$ orthogonal to $K_{\eps}$, i.e. 
$$
K_{\eps}^{\perp} = \left\{ u\in H^{1}_0(\Omega)\;:\; \int_{\Omega} \nabla PZ_{i,\eps}  \cdot \nabla u\, dx = \int_{\Omega} e^{U_{\eps}} Z_{i,\eps} u \,dx =0, \; i=0,1,2 \right\}. 
$$
Let $\pi$ and $\pi^{\perp}$ be the projections of $H^1_0(\Omega)$ respectively on $K_{\eps}$ and $K_{\eps}^{\perp}$.  Finally, we denote
$$
Y_\eps:=\{ f\in L^1(\Omega) \;:\; \|f \|_{\eps} <+\infty  \}.
$$ 

\begin{prop}\label{linear}
There exist $\eps_0>0$ and a constant $D_0>0$ such that 
\begin{equation}\label{linearest}
\|\ph\|_{H^1_0(\Omega)} + \|\ph\|_{L^\infty(\Omega)} \le  D_0 \| h \|_{\eps},
\end{equation}
for any $\eps \in (0,\eps_0)$, $\mu \in \mathcal U$, $\xi \in B(\xi_0,\sigma)$,  $h\in Y_\eps$ and  $\ph \in K_{\eps}^{\perp}$ satisfying
\begin{equation}\label{linearEq}
\pi^\perp\left\{ L \ph - (-\Delta)^{-1}h\right\}=0.
\end{equation} 
\end{prop}
\begin{proof}
We assume by contradiction that there exists $\eps_n \to 0$, $\mu_n\in \mathcal U$, $\xi_n \in B(\xi_0,\sigma)$,  $h_n\in Y_\eps$ and a solution $\ph_n\in K_{\eps_n}^{\perp}$ of \eqref{linearEq} such that 
$$
\frac{\|\ph_n\|_{H^1_0(\Omega)} + \|\ph_n\|_{L^\infty(\Omega)} }{\| h_n\|_{\eps_n}} \to +\infty. 
$$ 
Let $\delta_{n},\alpha_n,\beta_n$  be the parameters in Lemma \ref{Lemma parameters} corresponding to $\eps_n$, $\mu_n$ and $\xi_n$.  Let also $\rho_{0,n},\rho_{1,n}$, $\rho_{2,n}$ be defined as in \eqref{Radii}. We denote $\omega_n:= \omega_{\eps_n}$, $U_n := U_{\eps_{n}}$, $Z_{i,n}:= Z_{i,\eps_n}$ and $f_n:= f_{\eps_n}$.  W.l.o.g we can assume that $\|\ph_n\|_{H^1_0(\Omega)}+ \|\ph_n\|_{L^\infty(\Omega)}=1$ and $\|h_n\|_{\eps_n}\to 0$.  Since $\ph_n$ satisfies \eqref{linearEq}, there exist $c_{i,n} \in \R$, $i=0,1,2$, such that 
\begin{equation}\label{linearEq2}
-\Delta \ph_n  - \lambda f_n'(\omega_n )\ph_n  = h_n  + \sum_{i=0}^2 c_{i,n} e^{U_n} Z_{i,n}.
\end{equation}

\begin{Step} We have $c_{i,n} \to 0$ as $n\to +\infty$, $i=0,1,2$. 
\end{Step}
\medskip
Let  $\wt Z_{n}:=\wt Z_{\eps_n}$ be the function defined in {\eqref{Def Zbis}}. Testing equation \eqref{linearEq2} against $\wt Z_n$, we get
\begin{equation}\label{StL1}
\sum_{j=0}^2 c_{j,n} \int_{\Omega} e^{U_n}  Z_{j, n} \wt Z_{n} dx = \int_{\Omega} \nabla \wt  Z_n \cdot \nabla \ph_n   dx  -\int_{\Omega} \lambda f_n'(\omega_n)  \ph_n \wt Z_{n} dx - \int_{\Omega}  h_n \wt Z_n dx. 
\end{equation}
Since $\|\ph_n\|_{H^1_0(\Omega)}\le 1$ and $\ph_n \in K_{\eps_n}^\perp$, using Lemma \ref{Zbis} we get
$$ 
\int_{\Omega}  \nabla \wt Z_n \cdot \nabla \ph_n   dx  =  \int_{\Omega}  \nabla  Z_{0,n} \cdot  \nabla \ph_n   dx +o(1)=  {\underbrace{\int_{\Omega}  e^{U_n}Z_{0,n} \ph_n dx }_{=0}}  +o(1) = o(1),
$$
as $n\to +\infty$.  By Lemma \ref{Est L} and Lemma \ref{intGamma}, we find
\[\begin{split}
\int_{\Omega} \lambda f_n'(\omega_n) \ph_n \wt Z_n dx&  = \int_{B(\xi_n,\rho_{0,n})} e^{U_n} \ph_n  Z_{0,n} dx 
+O(\eps^2_n)+ O\left( \| \Gamma_\eps\|_{L^1(B(\xi_n,\rho_{1,n}) \setminus B(\xi_n,\rho_{0,n}))} \right) \\
&=    \underbrace{\int_{\Omega} e^{U_n} \ph_n  Z_{0,n} dx}_{=0} +o(1)  = o(1). 
\end{split}
\]
Finally, Lemma \ref{Zbis} and Lemma \ref{L1} give
$$
|\int_{\Omega}  h_n \wt Z_n dx| \le \|h_n\|_{L^1(\Omega)} \le C \|h_n\|_{\eps_n} =o(1). 
$$
Then \eqref{StL1} rewrites as 
\begin{equation}\label{StL2}
\sum_{j=0}^2 c_{j,n} \int_{\Omega} e^{U_n}  Z_{j, n} \wt Z_{n} dx = o(1). 
\end{equation}
With similar arguments, testing equation \eqref{linearEq2} against $PZ_{i,n}$ for $i=1,2$, we get that 
\begin{equation}\label{StL3}
\begin{split}
\sum_{j=0}^2 c_{j,n} \int_{\Omega} e^{U_n}  Z_{j, n}  P Z_{i, n} dx  & =   -\int_{\Omega} {\lambda f_n'(\omega_n)} \ph_n  P Z_{i,n} dx - \int_{\Omega}  h_n P Z_{i,n} dx\\
& =  { \underbrace{\int_{\Omega} e^{U_n} \ph_n Z_{i,n} dx }_{=0} +o(1)}=o(1).
\end{split}
\end{equation}
Note that, as in \eqref{orth},  we have
\[\begin{split}
\int_{\Omega} e^{U_n}  Z_{j, n}  \wt Z_{n} dx & = \int_{B(\xi_n,\rho_{0,n})} e^{U_n}  Z_{j, n}   Z_{0, n} dx + O\left( \int_{\R^2 \setminus B(\xi_n,\rho_{0,n})} e^{U_n}\right)  \\
& = \int_{ B(0,\frac{\rho_{0,n}}{\de_n})} e^{\ov U} Z_j Z_0 dy +o(1)\\
& =  \frac{8}{3}\pi \delta_{0j}  + o(1),
\end{split}\]
for $j= 0,1,2$. Similarly
\[\begin{split}
\int_{\Omega} e^{U_n}  Z_{j, n}  P Z_{i, n} dx & = \int_{\Omega} e^{U_n}  Z_{j, n}   Z_{i, n} dx + o(1) \\
 & = \frac{8}{3} \pi \delta_{ij}  + o(1),
\end{split}\]
for $i=1,2$, $j=0,1,2$. Then, \eqref{StL1} and \eqref{StL2} rewrite as 
\[
\sum_{j=0}^2 c_{j,n} (\delta_{ij} + o(1)) =o(1),
\]
which implies the conclusion. 

\begin{Step}
If $\widetilde h_{n} := h_n +{\left(\lambda f_n'(\omega_n)\chi_{B(\xi_n,\rho_{1,n})} - e^{U_n} \right)\ph_n}+ \sum_{j=0}^{2}c_{j,n} e^{U_n}Z_{j,n}$, then 
\begin{equation}\label{eq phi simple}
-\Delta \ph_n = e^{U_n} \ph_n + \lambda f_n'(\omega_n)\chi_{\Omega\setminus B(\xi_n,\rho_{1,n})}\ph_n + \wt h_n \quad \text{in } \Omega, \quad  \text{ and } \quad 
\|\wt h_n\|_{\eps_n} \to 0.
\end{equation}
\end{Step}

\medskip
Since $\|h_n\|_{\eps_n}\to 0$, $|Z_{i,n}|\le 1$,  and $\|\lambda f_n'(\omega_n)\chi_{B(\xi_n,\rho_{1,n})} - e^{U_n}\|_{\eps_n} \to 0$   by Lemma \ref{Est L}, it is sufficient to observe that $\|e^{U_n}\|_{\eps_n}=O(1)$ and  apply  Step 1.

\begin{Step}
There exists $\delta_n \ll \rho_n\le \rho_{0,n}$ such that, up to a subsequence, $\|\ph_n\|_{L^\infty(B(\xi_n,\rho_n))}\to  0$ as $n\to +\infty$.  
\end{Step}

Let us consider the sequence $\Phi_n(y):= \ph_n(\xi_n + \delta_n y) $, $y\in \frac{\Omega-\xi_n}{\delta_n}$. By  \eqref{eq phi simple} $\Phi_n$ satisfies 
$$
-\Delta \Phi_n  =  e^{\bar U} \Phi_n + \delta_n^2 \wt h_n(\xi + \delta_n \cdot) \qquad \text{ in } { B\Big(0,\frac{\rho_{1,n}}{\de_n}\Big)}. 
$$
We know that
$$
\left|e^{\bar U(y)}\Phi_n(y)\right|\le e^{\bar U(y)} \le \frac{8}{\mu^2},
$$
and, for  $y\in B(0,\frac{\rho_{0,n}}{\de_n})$, that
$$
\delta_n^2| \wt h_n(\xi + \delta_n y)|\le  \delta_n^2 j_{\eps_n}(\xi+ \delta_n y) \|\wt h_n\|_{\eps_n} = e^{\bar U(y)}(1+|\bar U(y)|^4)\|\wt h_n\|_{\eps_n} \le C \|\wt h_n\|_{\eps_n} \to 0.
$$
In particular $\Phi_n$ and $\Delta \Phi_n$ are uniformly bounded in $B(0,\frac{\rho_{0,n}}{\de_n})$. By standard elliptic estimates, we can find $\Phi_0\in { C(\R^2)\cap H^1_{loc}(\R^2)}$ and a sequence $R_n\to +\infty$,  $R_{n}\le \frac{\rho_{0,n}}{\de_n}$, such that, up to a subsequence,  $\|\Phi_n - \Phi_0\|_{L^\infty(B(0,R_n))}\to 0$.   Moreover,  $|\Phi_0|\le 1$ and $\Phi_0$  is a weak solution to
$$
- \Delta \Phi_0 =  e^{\bar U}\Phi_0 \quad \text{ in }\R^2. 
$$
According to Proposition \ref{5}, we must have  $\Phi_0 = \kappa_0 Z_{0}+ \kappa_1 Z_1+ \kappa_2 Z_2$, for some $\kappa_i \in \R$, $i=0,1,2$.  Keeping in mind \eqref{orth} and using that $e^{\bar U}\in L^1(\R^2)$, we obtain
\[
\begin{split}
0=\int_{\Omega} e^{U_n} Z_{i,n} \phi_n\, dx &  = \int_{\frac{\Omega-\xi_n}{\delta_n}} e^{\bar U} Z_{i} \Phi_n dy 
\\ & = \int_{B(0,R_n)} e^{\bar U} Z_{i} \Phi_n\; dy  + O\bra{ \int_{\R^2\setminus B(0,R_n)} e^{\bar U} dy } \\
   & \to  \frac{8}{3}\pi \kappa_i,
\end{split}
\]
for $i=0,1,2$. This implies $\kappa_i =0$, $i=0,1,2$. Then $\Phi_0 \equiv 0$ and  we get the conclusion with $\rho_n = \de_n R_n$.

\begin{Step} 
Up to a subsequence, $\xi_n \to \ov \xi \in \Omega$ and $\ph_n\to 0$ in $L^\infty_{loc}(\Omega\setminus \{\bar \xi \})$, as $n\to \infty$. 
\end{Step}
We know that $\ph_n$ satisfies  \eqref{eq phi simple} in $\Omega$. Since $|\ph_n|\le 1$, {$\|e^{U_n}\|_{L^{\infty}(\Omega\setminus B(\xi_n,{\rho_{1,n}}))}\to 0$,}  $\|h_n\|_{L^2(\Omega\setminus B(\xi,\rho_{1,n}))}\to 0$,  and  $\|f'_n(\omega_{n})\|_{L^\infty(\Omega\setminus B(\xi,\rho_{1,n}))} = O(1)$,  by ellpitic estimates we find that $\ph_n$ is bounded in $C^{0,\gamma}_{loc}(\ov \Omega\setminus \{\ov  \xi \})$, for some $\gamma \in (0,1)$. Therefore, there exists $\ph_0\in C(\ov \Omega)\cap H^1_0(\Omega)$, such that $\ph_n \to \ph_0$ locally uniformly on $\ov \Omega\setminus \{\ov \xi\}$ and weakly in $H^1_0(\Omega)$. Noting  that $\omega_{n}\to -u_0$ locally uniformly in $\ov \Omega\setminus \{\xi\}$ and that $f'_n$ is even, we see that $\ph_0$ satisfies  $\Delta \ph_0 + f'_0(u_0)\ph_0$ in $\Omega \setminus \{\bar \xi \}$. Actually, since $\ph_0,\Delta \ph_0 \in L^\infty(\Omega)$, $\ph_0$ is a weak solution of $\Delta \ph_0 + f'_0(u_0)\ph_0 =0$ in $\Omega$. Then, the  non-degeneracy of $u_0$ implies  $\ph_0 \equiv 0$. 

\begin{Step}  
$\|\ph_n\|_{L^\infty(\Omega)}\to 0$.
\end{Step}

\medskip
By Step 4, we can find a sequence $\sigma_n\ge \rho_{2,n}$ such that $\|\ph_n\|_{L^\infty(\Omega\setminus B(\xi_n,\sigma_n))} \to 0$ as $n\to +\infty$, up to a subsequence. Then, it is sufficient to show that $\|\ph_n\|_{L^\infty(A_n)}\to 0$, where $A_n:= B(\xi_n,\sigma_n) \setminus B(\xi_n,\rho_n)$ and $\rho_n$ is as in Step 3.   We can split $\ph_n = \ph_n^{(0)}+\ph_n^{(1)}+\ph_n^{(2)}+\ph_n^{(3)}$, where 
$$
\begin{cases}
\Delta \ph_n^{(0)}  = 0  & \text{ in }A_n,\\
\ph_n^{(0)} = \ph_n & \text{ on }\partial A_n,
\end{cases}
\qquad
\text{and}
\qquad 
\begin{cases}
- \Delta \ph_n^{(i)} =  f_{i,n} & \text{ in }A_n,\\
  \ph_n^{(i)}=0 & \text{ on }\partial A_n,
\end{cases}
\quad \text{ for } i=1,2,3,
$$
with 
$$
\begin{cases}
f_{1,n} := e^{U_n}\ph_n +\wt h_n \chi_{B(\xi_n,\rho_{0,n})},\\ 
f_{2,n} := \wt h_n \chi_{B(\xi_n,\rho_{1,n})\setminus B(\xi_n,\rho_{0,n})},\\
f_{3,n} := \wt h_n \chi_{B(\xi_n,\sigma_{n})\setminus B(\xi_n,\rho_{1,n})}+ \lambda f_n'(\omega_n) \chi_{B(\xi_n,\sigma_{n})\setminus B(\xi_n,\rho_{1,n})} \ph_n.\end{cases}$$  By the maximum principle   $$\|\ph_n^{(0)}\|_{L^\infty(A_n)} \le \|\ph_n\|_{L^\infty(\partial A_n)}\to 0.$$ 
Since
$$
|f_{1,n}| \le e^{U_n} + \|\wt h_n\|_{\eps_n} j_{\eps_{n}} \le j_{\eps_n} (1+o(1))\le 2 j_{\eps_n},
$$
we get that $|\ph_n^{(1)}|\le 2\psi_n$, where $\psi_n$ satisfies 
$$
\begin{cases}
-\Delta \psi_n = j_{\eps_n} & \text{ in }A_n\\
\psi_n = 0 & \text{ on }\partial A_n. 
\end{cases}
$$
Lemma \ref{barrier} implies $\|\psi_n\|_{L^{\infty}(A_n)} \to 0$, hence $\|\ph_n^{(1)}\|_{L^\infty(A_n)}\to 0$. Finally, since $|A_n|$ is uniformly bounded, elliptic estimates (see Corollaries \ref{CorStamp1} and \ref{CorStamp2}) give 
$$
\|\ph_n^{(2)}\|_{L^\infty(A_n)}\le \frac{C}{\alpha^2} \| f_{2,n} \|_{L^{1+\alpha^2}(A_n)}= \frac{C}{\alpha^2} \|\wt h_n \|_{L^{1+\alpha^2}(B(\xi_n,\rho_{1,n})\setminus B(\xi_n,\rho_{0,n}))} \le \|\wt h_n\|_{\eps_n} \to 0,
$$
and 
$$
{\|\ph_n^{(3)}\|_{L^\infty(A_n)}\le C \| f_{3,n} \|_{L^{2}(A_n)} = O ( \|h_n\|_{\eps_n})+ O(\sqrt{\sigma_n})\to 0}.
$$ 

\begin{Step} 
Conclusion of the proof.
\end{Step}

By Step 5, we have that $\|\ph_n\|_{H^{1}_0(\Omega)} = 1 - \|\ph_n\|_{L^\infty(\Omega)} \to 1$. But  \eqref{eq phi simple}  gives 
{\[\begin{split}
\|\ph_n\|_{H^1_0(\Omega)}^2  
&= \int_{\Omega} e^{U_n}\ph_n^2 \,dx + \int_{\Omega\setminus B(\xi,\rho_{1,n})} \lambda f'_n(\omega_n) \ph_n^2 \,dx + \int_{\Omega} \wt h_n \ph_n \,dx \\
& = O(\|\ph_n\|_{L^\infty(\Omega)}^2) + o(\|\ph_n\|_{L^2(\Omega)}) \to 0.
\end{split}\]
}
Then, we get a contadiction. 
\end{proof}

As a consequence we have that $\pi^\perp L$ is invertible on $K_{\eps}^{\perp}$. 

\begin{cor}\label{invertibility}
$\pi^\perp L:K_{\eps}^{\perp}\mapsto K_{\eps}^{\perp}$ is invertible.
\end{cor}
\begin{proof}
This follows by standard Fredholm theory. Indeed, for any $\eps>0$ the map $F(\ph):= \pi^\perp (-\Delta)^{-1}(  f'(\omega_{\eps})\ph ) $ defines a compact operator on $K_{\eps}^\perp$ (in fact on $H^1_0(\Omega)$). Then $\pi^{\perp}L= Id_{K_{\eps}^{\perp}} - F$ is a Fredholm operator of index $0.$ Proposition \ref{linear} implies that $\pi^\perp L$ is injective, hence it is invertible on $K_{\eps}^\perp$. 
\end{proof}

\section{The reduction to a finite dimensional problem}\label{Sec T} 

This section is devoted to reduce the problem to a finite dimensional one. More precisely, we prove: 
\begin{prop}\label{projsol}
There exist $\eps_0>0$ and  a  map $(\eps,\mu,\xi)\to \ph_{\eps,\mu,\xi}\in K_\eps^\perp \cap L^\infty(\Omega)$ defined in  $ (0,\eps_0)\times \mathcal U\times B(\xi_0,\sigma)$ and continuous with respect to $\mu$ and $\xi$,  such that for some $D>0$
\begin{equation}\label{norms phi}
\|\ph_{\eps,\mu,\xi}\|_{H^1_0}+\|\ph_{\eps,\mu,\xi}\|_{L^\infty} \le D\alpha^3,
\end{equation}
and
\begin{equation}\label{Eq proj}
\pi^{\perp} \Big\{L \ph_{\eps,\mu,\xi}  - (-\Delta)^{-1} (R + N(\ph_{\eps,\mu,\xi}))\Big\}=0,
\end{equation}
where the linear operator $L$ is defined in \eqref{DefL}, the error term $R$ is defined in \eqref{DefR} and the quadratic term $N$ is defined in \eqref{DefN}.
\end{prop}

\subsection{Estimates on \texorpdfstring{$N(\ph)$}{N(f)}}\label{Sec N}
For a function $\ph \in H^1_0(\Omega) \cap L^\infty(\Omega)$, let $N(\ph)$ be defined as in \eqref{DefN}, i.e. 
$$N(\ph)=N_{\eps,\mu,\xi }( \ph):=  \lambda \left( f_\eps(\omega_{\eps,\mu,\xi} + \ph) - f_\eps(\omega_{\eps,\mu,\xi}) - f_\eps'(\omega_{\eps,\mu,\xi}) \ph \right) .$$ 
Let us  estimate  $\|N(\ph)\|_{\eps}$, where $\|\cdot \|_{\eps}$ is defined as in \eqref{norm}.  Let us define 
\begin{equation}\label{Balpha}
\mathcal B_{\alpha}:=\{\ph \in L^\infty(\Omega)\;:\; \|\ph\|_{L^\infty(\Omega)} \le \alpha\}. 
\end{equation} 

\begin{lemma}\label{EstNDiff}
There exists $D_2>0$ such that 
$$
\|N(\ph_1) - N(\ph_2)\|_{\eps} \le D_2 \alpha^{-1} \bra{\|\ph_1\|_{L^\infty(\Omega)} +\|\ph_2\|_{L^\infty(\Omega)}}\|\ph_1-\ph_2\|_{L^\infty(\Omega)},
$$
for any $\ph_1,\ph_2\in \mathcal B_{\alpha}$. 
\end{lemma}
\begin{proof}
First, for any $x\in \Omega$ we can find $\theta_1= \theta_1(x)\in [0,1]$ such that 
\[\begin{split}
N(\ph_2) - N(\ph_1) 
& =  \lambda  \left( f_\eps(\omega_{\eps} + \ph_2)-f_\eps(\omega_{\eps} + \ph_1) - f_{\eps}'(\omega_{\eps}) (\ph_2-\ph_1) \right)\\
& =  \lambda \left( f_\eps'(\omega_{\eps}+\theta_1 \ph_2 + (1-\theta_1)\ph_1)(\ph_2-\ph_1) - f_{\eps}'(\omega_{\eps}) (\ph_2-\ph_1) \right) \\
& = \lambda \left( f_\eps'(\omega_{\eps}+ \ph_3) -f_\eps'(\omega_{\eps}) \right)(\ph_2-\ph_1),
\end{split}
\]
where $\ph_3:= \theta_1 \ph_2 + (1-\theta_1)\ph_1 $. Furthermore, there exists $\theta_2 = \theta_2(x)$ such that
\[
\begin{split}
f_\eps'(\omega_{\eps}+\ph_3)  & = f_\eps'(\omega_{\eps}) +  f_\eps''(\omega_{\eps}+ \theta_2 \ph_3)\ph_3.
\end{split}
\]
Thus,  we obtain
\begin{equation}\label{Ndiff}
\begin{split}
|N(\ph_1)-N(\ph_2)| &= \lambda |f_\eps''(\omega_{\eps}+ \theta_2 \ph_3)||\ph_3| |\ph_1-\ph_2| \\
&\le  \lambda |f_\eps''(\omega_{\eps}+ \theta_2 \ph_3)| \bra{\|\ph_1\|_{L^\infty(\Omega)}+\|\ph_2\|_{L^\infty(\Omega)}}\|\ph_1-\ph_2\|_{L^\infty(\Omega)}. 
\end{split}
\end{equation}
Then, in order to conclude the proof, we shall bound $\|f''_\eps(\omega_{\eps} +\theta_2 \ph_3)\|_\eps$. Note that, there exists a universal constant $C_0>0$ such that
$$
|f_\eps''(t)| \le C_0  (1+|t|^3)e^{t^2+|t|^{1+\eps}}, \quad \forall t\in \R.
$$
By Remark \ref{Rem beta} we have $\omega_\eps = O(\beta)=O(\alpha^{-1})$. Since $|\ph_3|\le |\ph_1|+|\ph_2|\le 2\alpha$, we get
\begin{equation}\label{EstN1}
(\omega_{\eps} +\theta_2 \ph_3)^2 \le  \omega_{\eps}^2 + 2|\omega_\eps| |\ph_3| + \ph_3^2 = \omega_{\eps}^2 + O(1).
\end{equation}
By convexity, we also  have
\begin{equation}\label{EstN2}
|\omega_{\eps} +\theta_2 \ph_3|^3 \le (|\omega_{\eps}|+|\ph_3|)^3 \le 4 (|\omega_{\eps}|^3 + |\ph_3|^3) \le 4 (|\omega_{\eps}|^3 + \alpha^3).
\end{equation} 
In $ B(\xi,\rho_{1})$ we have $\omega_{\eps} \ge c_0$ by Lemma \ref{positivity}, so that
\begin{equation}\label{EstN3}
(\omega_{\eps} + \theta_2 \ph_3)^{1+\eps} \le \omega_{\eps}^{1+\eps} \bra{1+ \frac{\alpha}{c_0}}^{1+\eps} = \omega_{\eps}^{1+\eps} + O(1).
\end{equation}
Clearly \eqref{EstN1}-\eqref{EstN3} yield the existence of a constant $C_1>0$ such that
$$
|f''_\eps(\omega_{\eps}+\theta_2 \ph_3)|  \le  C_1 \alpha^{-2} \omega_\eps  e^{\omega_{\eps}^2 + |\omega_{\eps}|^{1+\eps}} = C_1 \alpha^{-2} f_\eps(\omega_\eps),
$$
in $ B(\xi,\rho_{1})$. 
{ Arguing as in Lemma \ref{RinBrho0} (see \eqref{finBrho0}) we get
\begin{equation}\label{EstN4}
\lambda |f''_\eps(\omega_{\eps}+\theta_2 \ph)| \le C \alpha^{-1} j_{\eps} \quad \text{ in }B(\xi,\rho_0).
\end{equation}
Lemma \ref{finBrho1} and Lemma \ref{intGamma} yield
%
\begin{equation}\label{EstN5}
\lambda \|f''_\eps(\omega_{\eps}+\theta_2\ph)\|_{L^{1+\alpha^2}(B(\xi,\rho_1)\setminus B(\xi,\rho_0))} = O( \alpha^{-2}e^{-\frac{c_2}{\sqrt{\alpha}}}). 
\end{equation}
Finally, thanks to Lemma \ref{RinBrho2}}, we know that 
\begin{equation}\label{EstN6}
\lambda f''_\eps(\omega_{\eps}+\theta_2 \ph_3)=O(1) \quad \text{ in } \Omega \setminus B(\xi,\rho_1). 
\end{equation}
Thanks to \eqref{EstN4}-\eqref{EstN6} we infer
$$
\lambda \|f''_\eps(\omega_{\eps}+\theta_2 \ph_3)\|_{\eps} =O(\alpha^{-1}),
$$
and the conclusion follows from \eqref{Ndiff}. 
\end{proof}

\begin{rem}\label{EstN} Applying  Lemma \ref{EstNDiff} with $\ph_2=0$, we obtain that 
$$
\|N(\ph)\|_{\eps}\le D_2 \alpha^{-1} \|\ph\|_{L^\infty(\Omega)}^2,
$$ 
for any $\ph \in \mathcal B_\alpha$. 
\end{rem}

\begin{rem}\label{RemN} The proof of Proposition  \ref{EstNDiff} and Lemma \ref{RinBrho2} also shows that 
\[
\|N(\ph)\|_{L^\infty(\Omega\setminus B(\xi,\rho_1))} \le  { D_3  \|\ph\|_{L^\infty(\Omega)}^2},
\] 
for any $\ph \in \mathcal B_\alpha$. 
\end{rem}

\subsection{Proof of Proposition \ref{projsol}: a fixed point argument}
Let us consider the operator
\begin{equation}
\T= \T_{\eps,\mu,\xi}:=(\pi^{\perp}  L)^{-1}\pi^{\perp}\left[(-\Delta)^{-1}\Big( N(\ph)+ R\Big)\right]
\end{equation}
on the space $X:= K_{\eps}^{\perp}\cap L^\infty(\Omega)$, which is a Banach space with respect to the norm $$\|\cdot\|_{X}=\|\cdot\|_{H^1_0(\Omega)}+ \|\cdot\|_{L^\infty(\Omega)}.$$ 
 Let $D_1$ and $D_0$ be the constants defined in Proposition \ref{EstR} and Proposition \ref{linear}. Let us set
$$
E_\eps:=\{ \ph \in X\;:\;  \|\ph\|_{X} \le D_0(D_1+1)\alpha^3\}.
$$
 
 Proposition \ref{projsol} is an immediate consequence of the following result.

\begin{prop}
There exists $\eps_0>0$ such that, for any $\eps\in (0,\eps_0)$, $\mu \in \mathcal{U}$, $\xi \in B(\xi_0,\sigma)$,  $\T$ has a fixed point $\ph_{\eps,\mu,\xi}\in E_\eps$, which depends continuosly on  $\mu$ and $\xi$. 
\end{prop}
\begin{proof}
Since $E_\eps$ is a closed subspace of $X$ and $\T$ depends continuously on $\mu$ and $\xi$, it is sufficient to verifry that 
\begin{enumerate}
\item $\mathcal{T}$ maps $E_\eps$ into itself. 
\item $\mathcal{T}$ is a contraction, i.e. $||\mathcal{T}(\ph_1)-\mathcal{T}(\ph_2)||_{H^1_0(\Om)}
\le \theta||\ph_1-\ph_2||_{H^1_0(\Om)}$ for some positive constant $\theta<1$ and for any $\ph_1,\ph_2\in E_\e$.
\end{enumerate}
Then the conclusion follows by the contraction mapping theorem. 
\setcounter{Step}{0}\begin{Step}
$\mathcal{T}$ maps $E_\e$ into itself. 
\end{Step}
Let us denote   $C_0:= D_0( D_1+1)$.  Take $\ph \in E_\eps$ and set
$$
h(\ph):= R+ N(\phi).
$$
If $\eps $ is small enough, we have that $\alpha^2C_0\le 1$, so that $E_\eps\subseteq \mathcal B_\alpha$ (see \eqref{Balpha}).  By Proposition \ref{EstR} and Remark \ref{EstN} we get   
\[\begin{split}
\|h(\ph)\|_{\eps} &\le  \|R\|_{\eps} + \|N(\ph)\|_{\eps} \\
& \le  D_1 \alpha^3 + D_2 \alpha^{-1} \|\ph\|_{L^\infty(\Omega)}^2\\
& \le D_1 \alpha^3 + C_0^2 D_2 \alpha^5,
\end{split}
\]
for any $\ph \in E_\eps$. Then, if we take $\eps$ small enough so that ${C_0^2}D_2\alpha^2\le 1$, we get that 
$$
\|h(\ph)\|_{\eps}\le (D_1 +1)\alpha^3. 
$$
Since by definition
$$
\pi^\perp L(\mathcal T(\ph)) = \pi^{\perp} (-\Delta)^{-1} h(\ph),
$$
we have by Proposition \ref{linear} that  
\[\begin{split}
\|\mathcal T(\ph)\|_{X}  &\le  D_0 \|h(\ph)\|_{\eps} \le D_0  (D_1+1)\alpha^3,
\end{split}
\]
that is $\mathcal T(\ph)\in E_\eps$.
\begin{Step} 
 $\mathcal{T}$ is a contraction mapping in $E_\e$.
  \end{Step}

Let us take $\e$ small enough so that  $D_0 D_2 C_0\alpha^2\le \frac{1}{4}$ and $E_\eps\subseteq \mathcal B_\alpha$. By Propositions \ref{linear} and  \ref{EstNDiff}  we have
\[
\begin{split}
\|T(\ph_1)-T(\ph_2)\|_{X}  & \le D_0 \|h(\ph_1)-h(\ph_2)\|_{\eps}\\
& = D_0 \|N(\ph_1)-N(\ph_2)\|_{\eps} \\
& \le  D_0D_2 \alpha^{-1}( \|\ph_1\|_{L^\infty(\Omega)}+\|\ph_2\|_{L^\infty(\Omega)}) \|\ph_1-\ph_2\|_{L^\infty(\Omega)} \\
& \le  2 C_0 D_0 D_2\alpha^2 \|\ph_1-\ph_2\|_{L^\infty(\Omega)}\\
& \le \frac{1}{2}\|\ph_1-\ph_2\|_{L^\infty(\Omega)},
\end{split}
\]
for any $\ph_1,\ph_2\in E_\e$. Then, $\mathcal T$ is a contraction mapping on $E_\eps$.   
\end{proof}

\section{The reduced problem: proof of Theorem \ref{Trm precise} completed}\label{Sec Par}
Let $\ph_\eps := \ph_{\eps,\mu,\xi}$ be as in Proposition \ref{projsol}. By \eqref{Eq proj}, we can find  $\kappa_{\eps,i} = \kappa_{\eps,i}(\mu,\xi)$, $i=0,1,2$ (which depend continuously on $\mu$, and $\xi$), such that 
\begin{equation}\label{FinalEqPhi}
-\Delta \ph_\eps = \lambda f_\eps'(u_\eps){\ph_{\eps} }+R+ N(\ph_\eps) +    \sum_{j=0}^2 \kappa_{\eps,j} e^{U_{\eps}} Z_{\eps,j}.
\end{equation}
Equivalently, setting 
$u_\eps := \omega_\eps +\ph_\eps,
$
\begin{equation}\label{FinalEqu}
-\Delta u_\eps =\lambda f_\eps (u_\eps) + \sum_{j=0}^2 \kappa_{\eps,j} e^{U_{\eps}} Z_{\eps,j}.
\end{equation}
Our aim is to find the parameter $\mu = \mu(\eps)$ and the point $\xi = \xi(\eps)$ so that the $\kappa_{\eps,i}$'s are zero provided $\eps$ is small enough.
\medskip

 
 \begin{prop}\label{ki}
 It holds true that
\begin{equation}\label{FindMu}
\kappa_{0,\eps} = 6\pi \alpha^3 \left( 2- \log\Big( \frac{8}{\mu^2}\Big) +o(1) \right), \end{equation}
and
\begin{equation}\label{Expki} \begin{split}
\kappa_{i,\eps}   =   - \kappa_{0,\eps} a_{i,\eps} + \frac{3\mu}{2}  \delta   \DD{v_{\eps}}{x_i} (\xi) + O(\alpha \de),\ i=1,2
\end{split}
\end{equation} 
 as $\eps\to0$ uniformly with respect to $\mu\in\mathcal U$ and $\xi\in B(\xi_0,\sigma).$  Here, the $a_{i,\eps}$'s are continuous functions of $\mu$ and $\xi$ and  $ a_{i,\eps}= O(\alpha^2)$ uniformly for  $(\mu,\xi)\in \mathcal U \times B(\xi_0,\sigma)$. 
 \end{prop}

\begin{proof}
\setcounter{Step}{0}
\begin{Step} Let us prove that 
 \begin{equation}\label{EstKi}
\kappa_{i,\eps} = O(\alpha^3) \quad \text{ for } i =0,1,2
\end{equation}
and 
\begin{equation}\label{C1 bound}
\|\ph_\eps \|_{C^1(\ov \Omega\setminus B(\xi_0,2\sigma))} = O(\alpha^3). 
\end{equation}
\end{Step}

\medskip
First, since \eqref{norms phi} gives $\|\phi\|_{L^\infty(\Omega)}=O(\alpha^3)$, Proposition \ref{EstR}, Lemma \ref{L1}, Remark \ref{EstN} and Lemma \ref{Est L} yield   
$$
\|R\|_{L^1(\Omega)}=O(\alpha^3), \quad \|N(\ph_\eps)\|_{L^1(\Omega)} = O(\alpha^5), \quad \|\lambda f_\eps'(\omega_\eps) \ph_\eps \|_{L^1(\Omega)} =O(\alpha^3).
$$
Recalling that 
$$
\int_{\Omega} e^{U_n} Z_{j,n} PZ_{i,n} dx = \frac{8}{3}\pi \delta_{ij} +O(\delta), \text{ for } i,j = 0,1,2,
$$
by Lemma \ref{7} and \eqref{orth}, we get \eqref{EstKi} by testing equation \eqref{FinalEqPhi} with $PZ_{i,n}$, $i=0,1,2$. 

By Lemma \ref{Routside}, Remark \ref{RemN}, and Lemma \ref{Est L}, one has 
$$
\lambda f_\eps'(\omega_\eps) =O(1),  \quad   R =O(\alpha^3), \quad  { N(\ph_\eps) = O(\alpha^6)},
$$
uniformly in $\Omega\setminus {B(\xi,\frac{\sigma}{2})}$. Then
$$
{\|\Delta \ph_\eps \|_{L^\infty(\Omega \setminus B(\xi,\frac{\sigma}{2}))} + \|\ph_\eps\|_{L^\infty(\Omega)} = O(\alpha^3),}
$$
and we get \eqref{C1 bound} by  standard elliptic estimates.

\begin{Step} Proof of 
\eqref{FindMu}.
\end{Step}

\medskip
Let  $\wt Z_\eps $ be the function defined in \eqref{Def Zbis}. We shall test equation \eqref{FinalEqPhi} against $\wt Z_\eps$. With the same arguments of the proof of Proposition \ref{linear} (Step 1), we obtain 
$$
\int_{\Omega} \nabla \ph_\eps \cdot \nabla \wt Z_{\eps} \,dx =  \int_{\Omega} \nabla \ph_{\eps} \cdot \nabla Z_{0,\eps}\, dx + o(\|\ph_\eps\|_{H^1_0(\Omega)})  = o(\alpha^3).
$$
Moreover  
\[\begin{split}
\int_{\Omega} \lambda f_\eps'(\omega_{\eps}) \ph_{\eps} \wt Z_{\eps} dx & = \int_{B(\xi,\rho_0)} e^{U_\eps}Z_{0,\eps} \ph_\eps \, dx +O(\eps^2\alpha^3)+ O(\alpha^3 \|\Gamma_\eps\|_{L^1(B(\xi,\rho_1)\setminus B(\xi,\rho_0) )})\\
&= o(\alpha^3),  
\end{split}
\]
and 
\[\begin{split}
\int_{\Omega} e^{U_n} Z_{j,\eps} \wt Z_{\eps} dx &  = \int_{\R^2} e^{\bar U } Z_{j} Z_0 dy  +O\bra{\int_{\R^2 \setminus B(0,\frac{\rho_0}{\de})} e^{\bar U} dx }\\
& = \frac{8}{3}\pi \de_{ij} + O(\delta^2 \rho_{0}^{-2}). 
\end{split}
\]
By Lemma \ref{RinBrho0} and Lemma \ref{RinBrho1}, we get 
\[\begin{split}
\int_{\Omega} R \wt Z_n dx  & = \int_{B(\xi,\rho_0)} R Z_{0,n} dx +  O ( \|R\|_{L^1(B(\xi,\rho_1) \setminus B(\xi,\rho_0))}) \\
& =  \alpha^3 \int_{B(0,\rho_0)} e^{\bar U} \left( 2\bar U + \bar U^2 \right) Z_0 dy + O\left( \alpha^4 \int_{\R^2} e^{\bar U} (1+\bar U^4) dy \right) +  o ( \alpha^4) \\
& = 16\pi \alpha^3  \left( \log\Big( \frac{8}{\mu^2}\Big) - 2 \right)+O(\alpha^4).
\end{split}
\] 
Finally, we have that 
$$
\int_{\Omega} N(\ph) \wt Z_{\eps} dx = O(\|N(\ph)\|_{\eps}) =O(\alpha^5).
$$
Then, testing  \eqref{FinalEqPhi} against $\wt Z_\eps$ and using \eqref{EstKi}, one gets 
$$
0=  16\pi \alpha^3  \left( \log\Big( \frac{8}{\mu^2}\Big) - 2 \right)+ \frac{8}{3}\pi k_{0,\eps} + o(\alpha^3),
$$
from which we get \eqref{FindMu}. 
\begin{Step} Let us prove 
\begin{equation}\label{FindXi}
\sum_{j=0}^2\kappa_{j,\eps} \int_{\Omega} e^{U_\eps} Z_{j,\eps} \DD{u_\eps}{x_i} dx = - 8\pi  \alpha \DD{v_\eps}{x_i}(\xi)  + O(\alpha^2),\ i=1,2,
\end{equation}
\end{Step} 

\medskip
We multiply \eqref{FinalEqu} and  $\frac{ \partial u_{\eps}}{\partial x_i}$. Applying the  Pohozaev identity (see e.g. \cite[Proposition 2, Proof of Step 1]{Rey}), we obtain 
\begin{equation}\label{PID}
 - \frac{1}{2}\int_{\partial \Omega} \DD{u_\eps}{x_i} \frac{\partial u_{\eps}}{\partial \nu} \nu_i\, d\sigma = {\lambda } \int_{\Omega} f_\eps(u_\eps) \frac{ \partial u_{\eps}}{\partial x_i} dx + \sum_{j=0}^2 \kappa_{j,\eps} \int_{\Omega} e^{U_n} Z_{j,\eps} \frac{ \partial u_{\eps}}{\partial x_i} dx_i.
\end{equation}
Since $u_\eps = 0$ on $\partial \Omega$, the divergence theorem yields
\begin{equation}\label{Identity1}
\begin{split}
\int_{\Omega} f_\eps(u_\eps) \frac{ \partial u_{\eps}}{\partial x_i} dx &= \int_{\Omega} \frac{d}{dx_i} \left(\int_0^{u_\eps(x)} f_\eps(t) dt \right) dx  \\ &  = \int_{\partial \Omega} \nu_i \left(\int_0^{u_\eps(x)} f_\eps(t) dt \right) d\sigma=0.
\end{split}
\end{equation}
By   \eqref{C1 bound}, the definition of $u_\eps$ and $\omega_\eps$, Lemma \ref{Lemma weps & zeps},  Lemma \ref{LastRegion}, we have 
$$
\DD{u_\eps}{\nu} = -\DD{v_\eps}{\nu} +   \alpha \DD{}{\nu}(8\pi G_{\xi}-w_\eps ) + O(\alpha^2)
$$
on $\partial \Omega$. Thus, keeping in mind that $|\nabla v_\eps|$, $ |\nabla w_\eps|$ and $|\nabla G_{\xi}|$ are uniformly bounded on $\partial \Omega$ (see Lemma \eqref{Lemma veps} and \eqref{Lemma weps & zeps}) and that $\DD{u_\eps}{x_i} = \DD{u_\eps}{\nu}\nu_i$, we obtain  
\begin{equation}\label{Identity3}
\begin{split} 
\int_{\partial \Omega} \DD{u_\eps}{x_i}\DD{u_\eps}{\nu} \, d\sigma &= \int_{\partial \Omega} \DD{v_\eps}{x_i} \DD{v_\eps}{\nu}  \, d\sigma +  2 \alpha \int_{\partial \Omega} \DD{v_\eps}{x_i}\DD{}{\nu}(w_\eps-8\pi G_{\xi})   \, d\sigma  +O(\alpha^2).
\end{split}
\end{equation}
Applying the Pohozaev identity to $v_\eps$ and arguing as in \eqref{Identity1}, we get that 
\begin{equation}\label{Identity4}
\int_{\partial \Omega}  \DD{v_\eps}{x_i} \DD{v_\eps}{\nu} \,  d\sigma  = -2 \lambda  \int_{\Omega} f_\eps(v_\eps) \DD{v_\eps}{
 x_i} dx = 0. 
\end{equation}
Integrating by parts and noting that $-\Delta \DD{v_\eps}{x_i} = \lambda f'_\eps(v_\eps) \DD{v_\eps}{x_i} $ in $\Omega$, we get 
\[\begin{split}
\int_{\partial \Omega} \DD{v_\eps}{x_i}\DD{}{\nu}(w_\eps-8\pi G_{\xi})  d\sigma & =  \int_{ \Omega}\left( \DD{v_\eps}{x_i} \Delta w_\eps - w_\eps \Delta \DD{v_\eps}{x_i } \right) dx   + 8\pi \DD{v_\eps}{x_i}(\xi) \\
& \quad  + 8\pi \int_{\Omega} G_{\xi} \Delta \DD{v_\eps}{x_i} dx \\
& = \int_{\Omega}  \DD{v_\eps}{x_i} \underbrace{ \left( \Delta w_\eps + \lambda f'_\eps(v_\eps) w_\eps -8\pi \lambda f'_\eps(v_\eps)G_\xi   \right) }_{=0 \; \text{ by } \eqref{eqweps} } \, dx  + 8\pi \DD{v_\eps}{x_i}(\xi). 
\end{split}
\]
This together with \eqref{Identity3}-\eqref{Identity4} gives
\begin{equation}\label{Identity5}
\frac{1}{2}\int_{\partial \Omega} \frac{\partial u_\eps}{\partial x_i} \frac{\partial u_{\eps}}{\partial \nu} d\sigma = 8\pi \alpha \DD{v_\eps}{x_i}(\xi)  + O(\alpha^2).
\end{equation}
Finally,  \eqref{FindXi} follows by \eqref{PID}-\eqref{Identity1} and \eqref{Identity5}.

\begin{Step} For $i=1,2$,  $j=0,1,2$, we have  
\begin{equation}\label{coefficients}
\int_{\Omega} e^{U_\eps} Z_{j,\eps} \DD{u_\eps}{x_i} \, dx = - \frac{\alpha}{\delta} \left( \frac{16}{3\mu}\pi \delta_{ij} +O(\alpha^2) \right).
\end{equation}
\end{Step}

\medskip
For $i=1,2$ and $j=0,1,2$. Note that we have the identity  
$$
\DD{}{x_i} e^{U_\eps} Z_{j,\eps} =  \frac{e^{U_\eps}}{\delta \mu} \left( \delta_{ij}(Z_{0,\eps}+1)-\delta_{j0} Z_{i,\eps}  -  3 Z_{i,\eps}Z_{j,\eps} \right).
$$
Setting $ \Psi_{ij}:=\delta_{ij}(Z_{0}+1)-\delta_{j0} Z_{i}  -  3 Z_{i}Z_{j}$ and  applying the divergence theorem, we find
\[
\begin{split}
\int_{\Omega} e^{U_\eps} Z_{j,\eps} \frac{ \partial u_{\eps}}{\partial x_i} dx & =  - \int_{\Omega} u_{\eps} \frac{d}{d x_i} \left( e^{U_\eps} Z_{j,\eps} \right)dx  \\
 & = - \frac{1}{\delta \mu} \int_{\Omega} u_\eps  e^{U_\eps}  \left( \delta_{ij}(Z_{0,\eps}+1) -\delta_{j0} Z_{i,\eps} -  3 Z_{i,\eps}Z_{j,\eps}\right) dx \\ 
 & = -\frac{1}{\de \mu}   \int_{ \frac{\Omega -\xi}{\de} } u_\eps(\xi+ \de y)  e^{\bar U} \Psi_{ij} dy\\
 & = -\frac{1}{\de \mu}   \int_{ B(0,\frac{\sigma}{\de}) } u_\eps(\xi+ \de y)  e^{\bar U} \Psi_{ij} dy + O(\beta \de^{2}), 
\end{split}
\]
where in the last equality we used that
\begin{equation}\label{decay}
u_\eps = O(\beta) \qquad \text{and}  \qquad    e^{\bar U}  \Psi_{ij} = O ( |y|^{-5}),
\end{equation}
for $|y|\ge \frac{\sigma}{\de}$.  By Lemma \ref{expapproxsol} we have
\begin{equation}\label{ExpFinal}
u_{\eps}(\xi+ \delta y ) = \beta +\alpha \bar U(y) +O(\alpha^3)+O(\de |y|),
\end{equation}
for $y\in B(0,\frac{\sigma}{\de})$. Using again \eqref{decay}, we get that  
\[\begin{split}
\int_{ B(0,\frac{\sigma}{\de}) }   e^{\bar U}  \Psi_{ij} dy  & = \underbrace{\int_{ \R^2 }    e^{\bar U}  \Psi_{ij} dy }_{=0} +   O(\de^{3}).
\end{split}
\]
Similarly, we have
\[\begin{split}
\int_{ B(0,\frac{\sigma}{\de}) }  \ov U    e^{\bar U}  \Psi_{ij} dy  & = \int_{ \R^2 } \ov U    e^{\bar U}  \Psi_{ij} dy +   O(\beta^2 \de^{3})\\
& =   \frac{16}{3} \pi \delta_{ij} +   O(\beta^2 \de^{3}),
\end{split}
\]
and \eqref{coefficients} is proved. 
 \begin{Step} Proof of  \eqref{Expki}.
 \end{Step}
 
 \medskip
 Let us set $$
a_{ij,\eps}= a_{ij,\eps}(\xi,\mu):=-\frac{3\mu}{16\pi} \frac{\delta}{\alpha}\int_{\Omega} e^{U_\eps}Z_{j,\eps} \DD{u_\eps}{x_i} d\sigma.
$$
According to Step 4, we have  $a_{i0,\eps} = O(\alpha^2)$ if $i=1,2$. Moreover  the matrix $A = (a_{ij,\eps})_{i,j\in\{1,2\}}$ is invertible and its inverse $A^{-1}= (a^{ij}_\eps)_{ij\in \{1,2\}}$ satisfies 
$$a^{ij}_\eps =  \delta_{ij} +O(\alpha^2), \quad i,j=1,2.$$
Then \eqref{Expki} follows by \eqref{FindXi}, just setting
$$
a_{i,\eps}:=\sum_{j=1}^2 a^{ij}_\eps a_{0j,\eps}.
$$

\end{proof}

It is important to point out that \eqref{Expki} cannot be considered a precise uniform expansion of $\kappa_{i,\eps}$. Indeed, \eqref{FindMu} and the rough (but difficult to improve) estimate $a_{i,\eps} = O(\alpha^2)$ yield only  $\kappa_{0,\eps} a_{i,\eps} =O(\alpha^5)$. Since $\de\ll\alpha^5$  it is not possible to identify the leading term in the RHS of \eqref{Expki}.   However, it is clear that the term involving $\DD{v_\eps}{x_i}$ becomes dominant when $\kappa_{0,\eps}$ vanishes. This is enough for our argument. 

\bigskip
{\bf Proof of Theorem \ref{Trm precise} completed}\\
\begin{proof}
Let us consider the vector field 
$$B_\eps(\mu,\xi) = \left( \frac{1}{6\pi \alpha^3}\kappa_{0,\eps},  \frac{2}{3\de \mu}  \left(  \kappa_{1,\eps} + \kappa_{0,\eps}a_{1,\eps}\right),     \frac{2}{3\de \mu} \left(  \kappa_{2,\eps} + \kappa_{0,\eps}a_{2,\eps}\right) \right).$$
By construction, for any $\eps>0$, $B_\eps$ depends continuously on $\mu$ and $\xi$. Moreover,  thanks to  \eqref{FindMu}, \eqref{Expki} and Lemma \ref{Lemma veps}, we  have
$$
B_\eps \to \bar B(\mu,\xi) :=  \left(2- \log\Big( \frac{8}{\mu^2}\Big), \nabla u_0 (\xi) \right)
$$
as $\eps \to 0$, uniformly for $\mu \in \mathcal U$ and $\xi \in B(\xi_0,\sigma)$. 
By assumption \ref{A2}, $\bar B$ has  a $C^0$-stable zero at the point $(\mu_0,\xi_0)$, with $\mu_0 = \sqrt{8}e^{-1}$. Then, for $\eps$ small enough, there exist $\xi = \xi(\eps)\to \xi_0$, $\mu = \mu(\eps)\to\mu_0$ as $\eps\to0$ such that $B_\eps(\mu(\eps) ,\xi(\eps))=0 $. Clearly, this is equivalent to   $\kappa_{i,\eps,\mu(\eps),\xi(\eps)}=0$, $i =0,1,2$.  That concludes the proof.
\end{proof}

\appendix
\renewcommand\thesection{}
\section{Appendix A. The proof of Lemma \ref{Lemma parameters}}
\begin{proof}
The third equation in \eqref{parameters} allows to write $\delta$ as a function of $\alpha,\beta,\eps,\mu,\xi$:
$$
\log \frac{1}{\delta^2} = \frac{\beta}{2\alpha}+ \frac{V_{\eps,\alpha,\xi}(\xi)}{2\alpha} - \frac{c_{\mu,\xi}}{2},
$$
and the second equation in \eqref{parameters} gives $\alpha$ as a function of $\beta,\eps,\mu,\xi$:
$$\alpha=(2\beta+\beta^\eps+\eps\beta^\eps)^{-1}.$$
Then, (after a simple computation) it is sufficient to prove that  there exists $\beta = \beta(\eps,\mu,\xi)$   such that   
\begin{equation}\label{system1}
\begin{aligned}&\frac1\beta\left(\log \lambda  +\frac{c_{\mu,\xi}}{2}\right)+ 2\frac{\log \beta }\beta+ \underbrace{\left(\frac12\beta^{\eps} -u_0(\xi)\right)}_{:=\theta_\eps(\xi,\mu)}
-\left(V_{\eps,\alpha,\xi}(\xi)-u_0(\xi)\right)\\
&
+\frac{\log\left(2+\beta^{\eps-1}+\eps\beta^{\eps-1}\right)}\beta-\frac12\eps\beta^\eps-\frac12V_{\eps,\alpha,\xi}(\xi)\left(\beta^{\eps-1}+\eps\beta^{\eps-1}\right)=0.\end{aligned}
\end{equation}
 
Now, we choose $\beta^\eps:=2u_0(\xi)+\theta_\eps(\xi,\mu)$ with $\|\theta_\eps \|_{C^0(\overline{B(\xi_0,\sigma)}\times\overline{\mathcal U})}$ so small that
$$2u_0(\xi)+\theta_\eps(\xi,\mu)\ge\eta>1\ \hbox{in}\ \overline{B(\xi_0,\sigma)}\times\overline{\mathcal U}.$$ This is possible because of \eqref{sigma small}. With this choice we have 
$\frac1\beta=O\left(\eta^{-\frac1\eps}\right)$. It is easy to show that \eqref{system1} has a solution $\theta_\eps$ because of a simple fixed point argument. Indeed \eqref{system1} rewrites as $\theta_\eps=\mathcal T(\theta_\eps)$
where $\mathcal T$ is a contraction mapping on the ball $$\left\{\theta_\eps \in C^0(\overline{B(\xi_0,\sigma)}\times\overline{\mathcal U})\ : \|\theta_\eps \|_{C^0(\overline{B(\xi_0,\sigma)}\times\overline{\mathcal U})}\le \rho_\eps\right\},$$ where
$\rho_\eps:=\rho\min\left\{\frac1\eps\eta^{-\frac1\eps},\|v_\eps-u_0\|_{C^0(\overline\Omega)}\right\} $   for some $\rho>0$
and $\rho_\eps\to0$ as $\eps\to0.$ Here we use the expression of $V_{\eps,\alpha,\xi}(\xi)$ in \eqref{Veps} and (ii) of Lemma \ref{Lemma veps}.
\end{proof}

\section{Appendix B. A Stampacchia type estimate}
\renewcommand\thesection{\Alph{section}}
In this section we prove domain-independent estimates for solutions of the Poisson equation $-\Delta u = f$, under Dirichlet boundary conditions, with $f\in L^p(\Omega)$ and $p$ approaching $1$. Our strategy is based on the Stampacchia method.

\begin{lemma}[\cite{Stamp}, Lemma 4.1]\label{Stampacchia}
Let $\psi: \R^+\mapsto \R^{+}$ be a nonincreasing function. Assume that there exist $M>0, \gamma>0$, $\delta>1$ such that 
$$
\psi(h)\le \frac{M \psi(k)^\delta}{(h-k)^{\gamma}}\qquad \forall\; h>k>0. 
$$
Then $\psi(d)=0$, where $d=M^\frac{1}{\gamma}\psi(0)^{\frac{\de-1}{\gamma}}2^\frac{\de}{\de-1}$. 
\end{lemma}

Let $\Omega\subseteq \R^{2}$ be a bounded smooth domain. For any $q>1$, let $S_q(\Omega)$ be the Sobolev's constant for the embedding of $H^1_0(\Omega)$ in $L^q(\Omega)$, namely 
$$
S_q(\Omega)  = \inf_{u\in H^1_0(\Omega)} \frac{\| u\|_{H^1_0(\Omega)}}{\|u\|_{L^q(\Omega)}}. 
$$ 
It is known that $0<S_q(\Omega)<+\infty$ and that  (see {\cite{RW}} Lemma 2.2)
$$
\lim_{q\to +\infty} \sqrt{q} S_q(\Omega) = \sqrt{8\pi e}.  
$$

\begin{trm}
Let $\Omega$  be a bounded  smooth domain. For  $p>1$, $f\in L^p(\Omega)$, the unique solution $u\in H^1_0(\Omega)$ of the equation 
$-\Delta u = f$ satisfies 
\[
\|u\|_{L^\infty(\Omega)} \le  4S_{\frac{3p+1}{p-1}}(\Omega)^{-2}\|f\|_{L^p}|\Omega|^{\frac{p^2-1}{3p^2+p}}. 
\]
\end{trm}
\begin{proof}
We want to apply the previous lemma to the function 
$$
\psi(k):= |A_k|, \quad A_k:=\{x\in \Omega \;:\; |u(x)|>k\}. 
$$ 
For any $k>0$, let us consider the function 
$$
v_k(x):=\begin{cases}
0 &  |u(x)|\le k,\\
u(x)-k & u(x)>k,\\
-u(x)-k  & u(x)<-k.
\end{cases}
$$
Note that $v_k\in H^1_0(\Omega)$ and $|\nabla v_k| = |\nabla u|\chi_{A_k}$.  If we test the equation against $v_k$ we get 
\begin{equation}\label{App1}
\int_{\Omega} \nabla u \cdot \nabla v_k\, dx = \int_{\Omega} f v_k dx.
\end{equation}
For any $q\in (1,p)$ H\"older's inequality gives
\begin{equation}\label{App2}
\int_{\Omega} f v_k\, dx = \int_{A_k} f v_k \, dx\le \|f\|_{L^q(A_k)} \|v_k\|_{L^{\frac{q}{q-1}}(A_k)}  \le \|f\|_{L^p} |A_k|^\frac{p-q}{pq} \|v_k\|_{L^{\frac{q}{q-1}}(A_k)}.
\end{equation}
By Sobolev's inequality, we have that 
\begin{equation}\label{App3}
\int_{\Omega}\nabla u \cdot \nabla v_k\, dx = \int_{A_k} |\nabla v_k|^2 dx \ge S_{\frac{q}{q-1}}(\Omega)^{2} \|v_k\|_{L^\frac{q}{q-1}}^2.
\end{equation}
By \eqref{App1}-\eqref{App3}, we have 
$$
\|v_k\|_{L^\frac{q}{q-1}} \le S_{\frac{q}{q-1}}(\Omega)^{-2}\|f\|_{L^p} |A_k|^\frac{p-q}{pq}.
$$
Now, for any $h>k$, we have that $A_h\subseteq A_k$ and  $v_k \ge (h-k)$ in $A_h$, hence 
$$
\int_{\Omega} |v_k|^\frac{q}{q-1} dx = \int_{A_k} v_k^\frac{q}{q-1} dx \ge \int_{A_h} v_k^\frac{q}{q-1} dx   \ge (h-k)^\frac{q}{q-1}|A_h|.
$$
In conlcusion, we find 
$$
(h-k)|A_h|^{\frac{q-1}{q}} \le S_{\frac{q}{q-1}}(\Omega)^{-2}\|f\|_{L^p} |A_k|^\frac{p-q}{pq},
$$
or, equivalently, 
$$
\psi(h)\le \frac{S_{\frac{q}{q-1}}(\Omega)^{-\frac{2q}{q-1}}\|f\|_{L^p}^\frac{q}{q-1} \psi(k)^{\frac{p-q}{p(q-1)}}}{(h-k)^\frac{q}{q-1}}. 
$$
Then, we are in position to apply  Lemma \ref{Stampacchia} to $\psi$ with $M= S_{\frac{q}{q-1}}(\Omega)^{-\frac{2q}{q-1}}\|f\|_{L^p}^\frac{q}{q-1}$, $\gamma= \frac{q}{q-1}$, and $\delta= \frac{p-q}{p(q-1)}$.  For this, we need to impose that $\delta= \frac{p-q}{p(q-1)}$, that is $q< \frac{2p}{p+1}$. Note that $1<\frac{2p}{p+1}<p$. According to Stampacchia's Lemma, we have 
$$
\psi (d)=0 \qquad \text{ where } \qquad d=M^\frac{1}{\gamma}\psi(0)^{\frac{\de-1}{\gamma}}2^\frac{\de}{\de-1} = S_{\frac{q}{q-1}}^2\|f\|_{L^p}|\Omega|^{\frac{2p-q(p+1)}{pq}} 2^{\frac{p-q}{2p-q(p+1)}}. 
$$
This implies that
$$
\|u\|_{L^\infty(\Omega)} \le  S_{\frac{q}{q-1}}(\Omega)^{-2}\|f\|_{L^p}|\Omega|^{\frac{2p-q(p+1)}{pq}} 2^{\frac{p-q}{2p-q(p+1)}}. 
$$
This is true for any choice of $q\in (1,\frac{2p}{p+1})$. If we take for example $p$ the midpoint of  $(1,\frac{2p}{p+1})$, that is $q= \frac{1}{2}+\frac{p}{p+1}= \frac{3p+1}{2(p+1)}$, then we find that 
$$
\frac{q}{q-1} = \frac{3p+1}{p-1}, \quad \frac{2p-q(p+1)}{pq} = \frac{p^2-1}{3p^2+p},  \quad 
\frac{p-q}{2p-q(p+1)} = \frac{2p+1}{p+1}\le 2,
$$
and we get the conclusion.
\end{proof}

\begin{cor}\label{CorStamp1}
Given $K>0$ and $p>1$, there exists a constant $C=C(K,p)$ such that, for any domain $\Omega\subseteq \R^2$ with $|\Omega|\le K$ and any $f \in L^p(\Omega)$ the unique solution $u\in H^1_0(\Omega)$ of $-\Delta u = f$ satisfies 
$$
\|u\|_{L^\infty(\Omega)} \le C \|f\|_{L^p(\Omega)}. 
$$ 
\end{cor}

\begin{cor}\label{CorStamp2}
Given $K>0$, there exist $p_0=p_0(K)$ and $C=C(K)$ such that, for any $1<p<p_0$, any domain $\Omega\subseteq \R^2$ with $|\Omega|\le K$, and any $f \in L^p(\Omega)$, the unique solution $u\in H^1_0(\Omega)$ of $-\Delta u = f$ satisfies 
$$
\|u\|_{L^\infty(\Omega)} \le \frac{C}{p-1} \|f\|_{L^p(\Omega)}. 
$$ 
\end{cor}

\end{document}